\documentclass[11pt,twoside]{article}
\usepackage{amssymb,amsmath,amsthm,amsfonts,mathrsfs,hyperref}
\usepackage{times}
\usepackage{enumerate}
\usepackage{cite,titletoc}
\usepackage{color}
\usepackage[toc,page,title,titletoc,header]{appendix}

\allowdisplaybreaks

\def\cC{\mathcal C}

\def\cF{\mathcal F}

\def\cI{\mathcal I}

\def\cK{\mathcal K}
\def\cL{\mathcal L}

\def\cN{\mathcal N}

\def\cQ{\mathcal Q}

\def\cS{\mathcal S}
\def\cT{\mathcal T}
\def\cU{\mathcal U}

\def\cX{\mathcal X}
\def\cY{\mathcal Y}

\def\sF{\mathscr F}
\def\sL{\mathscr L}
\def\sT{\mathscr T}
\def\N{\mathop{\mathbb N\kern 0pt}\nolimits}
\def\Z{\mathop{\mathbb Z\kern 0pt}\nolimits}
\def\Q{\mathop{\mathbb Q\kern 0pt}\nolimits}
\def\R{\mathop{\mathbb R\kern 0pt}\nolimits}
\def\T{\mathop{\mathbb T\kern 0pt}\nolimits}
\def\C{\mathop{\mathbb C\kern 0pt}\nolimits}

\def\ds{\displaystyle}

\def\supp{\mathop{\rm supp}\nolimits}

\def\p{\partial}

\def\ve{\varepsilon}

\def\ls{\lesssim}

\newcommand{\w}[1]{\langle {#1} \rangle}
\def\id{\textbf{1}}

\hypersetup{colorlinks=true,linkcolor=blue,citecolor=red,urlcolor=cyan}

\theoremstyle{plain}
\newtheorem{theorem}{Theorem}[section]
\newtheorem{proposition}[theorem]{Proposition}
\newtheorem{lemma}[theorem]{Lemma}

\theoremstyle{definition}
\newtheorem{definition}[theorem]{Definition}

\newtheorem{remark}{Remark}[section]

\numberwithin{equation}{section}

\title{Long time solutions of quasilinear Klein-Gordon equations with small weakly decaying initial data}

\author{Hou Fei$^{1,*}$ \qquad
  Yin Huicheng$^{2, } $\footnote{Hou Fei (\texttt{fhou$@$nju.edu.cn}) and
    Yin Huicheng (\texttt{huicheng$@$nju.edu.cn}, \texttt{05407$@$njnu.edu.cn}) are supported by the NSFC (No.11731007, No.12101304).
    Hou Fei is also supported by the NSF of Jiangsu Province (No. BK20210170) and Shuangchuang Program of Jiangsu Province (JSSCBS20210008).
    In addition, Yin Huicheng is supported by the National key research and development program of China (No.2020YFA0713803).}\\
    [12pt] {\small 1. Department of Mathematics, Nanjing University, Nanjing 210093, China}\\
  {\small 2. School of Mathematical Sciences and Mathematical Institute, }\\
  {\small Nanjing Normal University, Nanjing 210023, China}}

\begin{document}

\date{}
\maketitle
\thispagestyle{empty}

\begin{abstract}
It is well known that for the quasilinear Klein-Gordon equation with quadratic nonlinearity and sufficiently decaying small initial data,
there exists a global smooth solution if the space dimensions $d\ge2$.
When the initial data are of size $\varepsilon>0$ in the Sobolev space, for the semilinear Klein-Gordon equation
satisfying the null condition, the authors in the article (J.-M. Delort, Daoyuan Fang, Almost global existence for solutions of semilinear Klein-Gordon equations with small weakly decaying Cauchy data, Comm. Partial Differential Equations 25 (2000), no. 11-12, 2119--2169)
prove that the solution exists in time $[0,T_\varepsilon)$ with $T_\varepsilon\ge Ce^{C\varepsilon^{-\mu}}$ ($\mu=1$ if $d\ge3$,
$\mu=2/3$ if $d=2$).
In the present paper, we will focus on the general quasilinear Klein-Gordon equation without the null condition and
further show that the existence time of the solution can be improved to $T_\varepsilon=+\infty$ if $d\ge3$ and $T_\varepsilon\ge e^{C\varepsilon^{-2}}$ if $d=2$. In addition, for $d=2$ and any fixed number $\alpha>0$, if the weighted $L^2$ norm of the initial data
with the weight $(1+|x|)^\alpha$ is small, then the solution exists globally and scatters to a free solution.
The arguments are based on the introduction of a good unknown, the Strichartz estimate, the weighted $L^2$-norm
estimate and the resonance analysis.

\vskip 0.2 true cm


\qquad \quad weighted $L^2$-norm, resonance analysis

\vskip 0.2 true cm
\noindent
\end{abstract}

\vskip 0.5 true cm


\section{Introduction}
In the paper, we are concerned with the Cauchy problem of the quasilinear Klein-Gordon equation
\begin{equation}\label{KG}
\left\{
\begin{aligned}
&\Box u+u=F(u,\p u,\p^2u),\quad(t,x)\in[1,\infty)\times\R^d,\\
&(u,\p_tu)(1,x)=\ve(u_0,u_1)(x),
\end{aligned}
\right.
\end{equation}
where $\Box=\p_t^2-\Delta$, $\Delta=\sum_{j=1}^d\p_j^2$, $x=(x^1,\cdots,x^d)\in\R^d$, $d\ge2$, $\p_0=\p_t$, $\p_j=\p_{x^j}$ for $j=1,\cdots,d$, $\p_x=(\p_1,\cdots,\p_n)$, $\p=(\p_0,\p_x)$, and $\ve>0$ is sufficiently small.
The smooth nonlinearity $F(u,\p u,\p^2u)$ is quadratic and is linear in $\p^2u$.

Our main results can be stated as follows.

\begin{theorem}\label{thm1}
Let $d\ge2$ and $N\ge2d+[d/2]+6$.
There are two positive constants $\ve_0$ and $\kappa$ such that for any $\ve\in(0,\ve_0)$, if $(u_0,u_1)$ satisfies
\begin{equation}\label{initial:data}
\|u_0\|_{H^{N+1}(\R^d)}+\|u_1\|_{H^N(\R^d)}\le1,
\end{equation}
then \eqref{KG} admits a unique solution $u\in C([1,T_\ve),H^{N+1}(\R^d))\cap C^1([1,T_\ve),H^N(\R^d))$, where $T_\ve=\infty$ if $d\ge3$ and $T_\ve=e^{\kappa/\ve^2}$ if $d=2$.
\end{theorem}

\begin{theorem}\label{thm2}
Assume $d=2$, $N\ge12$ and $\alpha\in(0,1/5)$. There is a positive constant $\ve_0$ such that for any $\ve\in(0,\ve_0)$, if $(u_0,u_1)$ satisfies
\begin{equation}\label{initial:d=2}
\|u_0\|_{H^{N+1}(\R^2)}+\|u_1\|_{H^{N}(\R^2)}
+\|\w{x}^\alpha\Lambda u_0\|_{L^2(\R^2)}+\|\w{x}^\alpha u_1\|_{L^2(\R^2)}\le1,
\end{equation}
where $\w{x}:=\sqrt{1+|x|^2}$, $\Lambda:=(1-\Delta)^{1/2}$, then \eqref{KG} has a unique global solution $u\in C([1,\infty),H^{N+1}(\R^2))$
$\cap C^1([1,\infty),H^{N}(\R^2))$.
In addition, the solution $u$ scatters to a free solution: there exists $(u_0^\infty,u_1^\infty)\in H^1(\R^2)\times L^2(\R^2)$,
and denote by $u^\infty$ the solution to the linear Klein-Gordon equation with initial data $(u_0^\infty,u_1^\infty)$ at time $t=1$,
then
\begin{equation}\label{scatter}
\lim_{t\rightarrow+\infty}\sum_{j=0}^1\|\p_t^j(u(t)-u^\infty(t))\|_{H^{1-j}}=0.
\end{equation}
\end{theorem}

\begin{remark}
We point out that Theorem \ref{thm1} extends the results in \cite{DF00} through the following three aspects:
First, the more general quasilinear case rather than only the semilinear case is studied.
Second, the requirement on the null condition of $F(u,\p u,\p^2u)$ is removed.
Third, the lifespan of the existence of the solution is improved.
\end{remark}

\begin{remark}
By the same method as in Theorem \ref{thm1}, we can get the result with the existence time $T_{\ve}=O(\frac{1}{\ve^4})$ for $d=1$, which improves the lifespan $T_{\ve}=O(\frac{1}{\ve^4 |\ln\ve|^6})$ in \cite{Delort97}, see Remarks \ref{Strichartz:d=1} and \ref{thm1pf:d=1} for details.
\end{remark}

\begin{remark}
The norm $\|\w{x}^\alpha\Lambda u_0\|_{L^2(\R^2)}$ in \eqref{initial:d=2} can be replaced by $\|\w{x}^\alpha u_0\|_{L^2(\R^2)}$.
In fact, due to the interpolation between $\|\w{x}^\alpha u_0\|_{L^2(\R^2)}\le1$ and $\|u_0\|_{H^{N+1}(\R^2)}\le1$, one has $\|\w{x}^{\frac{N\alpha}{N+1}}\Lambda u_0\|_{L^2(\R^2)}\le1$.
\end{remark}

\begin{remark}
If $\alpha\ge1/5$ in \eqref{initial:d=2}, then
$\|\w{x}^{1/6}\Lambda u_0\|_{L^2(\R^2)}+\|\w{x}^{1/6}u_1\|_{L^2(\R^2)}
\le\|\w{x}^\alpha\Lambda u_0\|_{L^2(\R^2)}+\|\w{x}^\alpha u_1\|_{L^2(\R^2)}$
holds and then the result of Theorem \ref{thm2} is true.
\end{remark}

\begin{remark}
Theorem \ref{thm1} and \ref{thm2} can be applied to extend the global perturbed
solutions of the  3D and 2D irrotational electron Euler-Poisson systems in \cite{GMP13,Guo98,IP13,LW14}
with the analogous small decaying data of \eqref{initial:data} and \eqref{initial:d=2}.
Note that the initial data in \cite{GMP13,Guo98} are required to have compact supports.
\end{remark}

\begin{remark}
We can also deal with the fully nonlinear quadratic case that $F(u,\p u,\p^2u)$ is not linear in $\p^2u$ and Theorems \ref{thm1}-\ref{thm2} still hold, see Remark \ref{fully:nonl} in Appendix B.
\end{remark}

We now recall some basic results on the nonlinear Klein-Gordon equation
\begin{equation}\label{NLKG}
\left\{
\begin{aligned}
&\Box u+m^2u=F(u,\p u,\p^2u),\quad(t,x)\in[1,\infty)\times\R^d,\\
&(u,\p_tu)(1,x)=\ve(u_0, u_1)(x),
\end{aligned}
\right.
\end{equation}
where $m\neq0$, $x\in\Bbb R^d$ ($d\ge 1$).

\vskip 0.1 true cm

{\bf $\bullet$} The cases of  $(u_{0},u_{1})\in C^{\infty}(\Bbb R^d)$
with suitably rapid decay at infinity or $(u_{0},u_{1})\in H^{s+1}(\R^d)\times H^s(\R^d)$

\vskip 0.1 true cm

When $d\ge2$, it is well known that problem \eqref{NLKG} with rapidly decaying $(u_0, u_1)$ has a global
smooth solution, see \cite{Klainerman85,OTT96,Shatah85,ST93}.
When $d=1$ and the nonlinearity $F$ satisfies the null condition, the author in \cite{Delort01}
establishes the global existence of \eqref{NLKG} for the rapidly decaying $(u_0, u_1)$.

If  $(u_0, u_1)\in H^{s+1}(\R^d)\times H^s(\R^d)$ with integer $s>(d+3)/2$ and the corresponding  semilinear $F=F(u,\p u)$ satisfies the null condition, then the lifespan $T_{\ve}$ of the solution $u$ to \eqref{NLKG} fulfills at least $T_\ve\ge Ce^{C/\ve}$ for $d\ge3$, $T_\ve\ge Ce^{C\ve^{-2/3}}$ for $d=2$ and $T_\ve\ge C\ve^{-4}|\ln\ve|^{-6}$ for $d=1$, respectively, where $C>0$ is a constant, see \cite{Delort97,DF00}. In addition, the author in \cite{Stingo18} proves the global existence of  \eqref{NLKG} with mildly decaying $(u_0, u_1)$ for $d=1$.

\vskip 0.1 true cm

{\bf $\bullet$} The cases of the periodic initial data $(u_{0},u_{1})$

\vskip 0.1 true cm

For $d=1$ and $F=F(x,u)$, the results in \cite{Bambusi03,BG06,Bourgain96} show that for any $M>0$ and the number $m$ except a subset
of zero measure in $\Bbb R$, when $(u_{0},u_{1})\in H^{s_1+1}(\T)\times H^{s_1}(\T)$ with $s_1$ depending on $M$, the solution
to \eqref{NLKG} exists for time $t\in[0,C_M\ve^{-M}]$ with  $C_M>0$ being some constant.

For $d\ge1$ and $F=F(u,\p u)$, the author in \cite{Delort98} has proved that if $F$ vanishes of order $r\ge2$ at $0$, then the lifespan $T_\ve$
of the periodic solution to problem \eqref{NLKG} satisfies at least $T_\ve\ge C\ve^{-2}$ for $r=2$ and
$T_\ve\ge C\ve^{-(r-1)}|\ln\ve|^{-(r-3)}$ for $r\ge3$.
For $d\ge2$ and $F=F(u)$ (even for $F(x,u)$), the result in \cite{Delort09} shows that when $F$ vanishes of order $r\ge2$ at $0$, for any $A>1$, there is $s_2>0$ such that problem \eqref{NLKG} has a unique solution $u\in C([0,T_\ve],H^{s_2+1}(\T^d))\cap C^1([0,T_\ve],H^{s_2}(\T^d))$ with $T_\ve\ge C\ve^{-(r-1)(1+2/d)}|\ln\ve|^{-A}$.
For more general nonlinearity $F=F(u,\p u,\p^2u)$, one can see \cite{DS04}.

\vskip 0.1 true cm

{\bf $\bullet$} The cases of the partial periodic initial data $(u_{0},u_{1})$ defined in $\Bbb R^{d_1}\times \T^{d_2}$ ($d_1+d_2=d$)

\vskip 0.1 true cm

For $1\le d_1\le4$ and $d_2=2$, the authors in \cite{HV18} study the small data scattering of the energy critical nonlinear Klein-Gordon equation $\Box u+u=\pm|u|^\frac{4}{d_1}u$ with initial data in $H^1(\R^{d_1}\times\T^{d_2})\times L^2(\R^{d_1}\times\T^{d_2})$.
The large data scattering of the defocusing nonlinear Klein-Gordon equation on $\R^{d_1}\times\T$ with $1\le d_1\le4$ in the subcritical case
has also been established in \cite{FH20}.

For the general nonlinearity $F=F(u,\p u,\p^2u)$ in \eqref{NLKG}, the authors in \cite{LTY22a,LTY22b,TY21} prove that problem \eqref{NLKG} with
$(u_0,u_1)$ defined on $\R^3\times\T$ or $\R^2\times\T$ admits a global solution, respectively.

\vskip 0.1 true cm

Next we give some comments on the proofs of Theorem \ref{thm1} and \ref{thm2}.
Note that for the weakly decaying initial data in \eqref{KG}, it is hard to get such a dispersive estimate of the solution $v$ to the linear Klein-Gordon equation $\Box v+v=0$ with $(v,\p_t v)(1,x)=(u_0,u_1)(x)$
\begin{equation}\label{intro:disp}
\|(\p_tv,\Lambda v)\|_{L^\infty(\R^d)}\le Ct^{-d/2}\|\Lambda^n(u_0,u_1)\|_{L^1(\R^d)},
\end{equation}
where $C$ and $n$ are some positive constants. The reason is that the $L^1$ norm of the right hand side
in \eqref{intro:disp} can become infinity since $\|u_j\|_{L^1(\R^d)}$ $(j=0,1)$ is controlled by $\|\w{x}^{d/2+}u_j\|_{L^2(\R^d)}$
and the latter is generally unbounded by $u_j\in H^{N+1-j}(\R^d)$. It is pointed out that the inequality \eqref{intro:disp} plays a key role
in \cite{LW14}, \cite{Shatah85,ST93} and so on.
Instead of \eqref{intro:disp}, we will employ the following Strichartz estimate
\begin{equation}\label{intro:Stric}
\|(\p_tv,\Lambda v)\|_{L^2([1,t])L^\infty(\R^d)}
\le\left\{
\begin{aligned}
&C\|\Lambda^n(u_0,u_1)\|_{L^2(\R^d)},&&d\ge3,\\
&C\ln^{1/2}t\|\Lambda^n(u_0,u_1)\|_{L^2(\R^2)},&&d=2.
\end{aligned}
\right.
\end{equation}
On the other hand, in order to apply the Strichartz estimate in the higher order energy estimates
of problem \eqref{KG},  the normal form method in \cite{Shatah85} which transforms the quadratic nonlinearity $F$ into
a cubic term can not be directly used due to the resulting loss of solution regularities in the transformation process.
To overcome this difficulty, we will carry out a careful resonance analysis similar to that in \cite{IP13,Zheng19} and
introduce a good unknown. At this time, the related cubic nonlinearity can be bounded by
$L_t^\infty H^N(\R^d)\times L_t^2W^{N',\infty}(\R^d)\times L_t^2W^{N',\infty}(\R^d)$ norms of the solution $u$
($N'$ is an integer) and further the energy estimates are derived.
Then Theorem \ref{thm1} is shown.
To prove Theorem \ref{thm2} for $d=2$, inspired by \cite{Zheng22}, we will establish
a kind of weighted Strichartz estimates instead of \eqref{intro:Stric}.
It is noticed that if $L^2([1,t])L^\infty(\R^2)$ is replaced by $L^p([1,t])L^\infty(\R^2)$ with any $p>2$ in \eqref{intro:Stric}, then \eqref{intro:Stric} holds without the factor $\ln^{1/2}t$.
In addition, integrating the dispersive estimate \eqref{intro:disp} in time yields
\begin{equation}\label{intro:Stric1}
\|s^{1/2}(\p_tv,\Lambda v)\|_{L^p([1,t])L^\infty(\R^2)}\le C\|\w{x}^{1+}\Lambda^n(u_0,u_1)\|_{L^2(\R^2)}
\end{equation}
provided that the right hand side of \eqref{intro:Stric1} is bounded.

It follows from the interpolation between \eqref{intro:Stric} and \eqref{intro:Stric1} that there is $\beta\in(0,\alpha)$ such that
\begin{equation}\label{intro:Stric2}
\|s^{\beta/2}(\p_tv,\Lambda v)\|_{L^p([1,t])L^\infty(\R^2)}\le C\|\w{x}^\alpha\Lambda^n(u_0,u_1)\|_{L^2(\R^2)}.
\end{equation}
Choosing $p>2$ in \eqref{intro:Stric2} such that $\|s^{-\beta/2}\|_{L^{2p/(p-2)}([1,t])}<\infty$ and then
\begin{equation}\label{intro:Stric3}
\begin{split}
\|(\p_tv,\Lambda v)\|_{L^2([1,t])L^\infty(\R^2)}&\le\|s^{-\beta/2}\|_{L^{2p/(p-2)}([1,t])}
\|s^{\beta/2}(\p_tv,\Lambda v)\|_{L^p([1,t])L^\infty(\R^2)}\\
&\le C\|\w{x}^\alpha\Lambda^n(u_0,u_1)\|_{L^2(\R^2)}.
\end{split}
\end{equation}
With this improved Strichartz estimate, the energy estimate as in the proof of Theorem \ref{thm1} can be established.
In addition, to complete the proof of Theorem \ref{thm2}, the remain task is to control the weighted $L^2$ norm in the right hand side of the resulting energy estimate due to the appearance of nonlinearity $F$ in \eqref{KG}.
For this purpose, both the dyadic decompositions in the frequency space
and in the Euclidean physical space $\R^2$ will be adopted. Together with the precise localized dispersive estimate and
Strichartz estimate, we can close the arguments on the weighted $L^2$ norm estimate of solution.

The paper is organized as follows.
In Section 2, some preliminaries such as the linear dispersive estimate, Strichartz estimates and paradifferential calculus are given.
By introducing a good unknown and utilizing resonance analysis, the higher order energy estimates for problem \eqref{KG}
will be established in Section 3.
In Section 4, the lower order energy estimates of \eqref{KG} are obtained and then Theorem \ref{thm1} is proved.
In Section 5, we will close the weighted $L^2$ norm estimate of solution and finish the proof of Theorem \ref{thm2}.
In addition, the estimates of some related multilinear Fourier multipliers are given in Appendix A.
Meanwhile, a basic reformulation of the good unknown is derived in Appendix B.

\section{Preliminaries}
\subsection{Linear dispersive estimate and Strichartz estimate}
For the function $f(x)$ on $\R^d$, define its Fourier transformation as
\begin{equation*}
\hat f(\xi):=\sF_xf(\xi)=\int_{\R^d}e^{-ix\cdot\xi}f(x)dx.
\end{equation*}
Choose a smooth cutoff function $\psi: \R\rightarrow[0,1]$, which equals 1 on $[-5/4,5/4]$ and vanishes outside $[-8/5,8/5]$, set
\begin{equation*}
\begin{split}
&\psi_k(x):=\psi(|x|/2^k)-\psi(|x|/2^{k-1}),\quad k\in\Z,k\ge0,\\
&\psi_{-1}(x):=1-\sum_{k\ge0}\psi_k(x)=\psi(2|x|),
\quad\psi_I:=\sum_{k\in I\cap\Z\cap[-1,\infty)}\psi_k,
\end{split}
\end{equation*}
where $I$ is any interval of $\R$.
Let $P_k$ be the Littlewood-Paley projection onto frequency $2^k$
\begin{equation*}
\sF(P_kf)(\xi):=\psi_k(\xi)\sF f(\xi),\quad k\in\Z,k\ge-1.
\end{equation*}
In addition, for any interval $I$, $P_I$ is defined by
\begin{equation*}
P_If:=\sum_{k\in I\cap\Z\cap[-1,\infty)}P_kf.
\end{equation*}

\begin{lemma}[Linear dispersive estimate]\label{lem:disp}
For any function $f$, integer $k\ge-1$ and $t\ge1$, it holds that
\begin{equation}\label{disp:estimate}
\|P_ke^{\pm it\Lambda}f\|_{L^\infty(\R^d)}\ls2^{k(d/2+1)}t^{-d/2}\|P_kf\|_{L^1(\R^d)},
\end{equation}
where and below for the non-negative quantities $f$ and $g$,
$f\ls g$ means $f\le Cg$ with $C$ being a generic positive constant.
\end{lemma}
\begin{proof}
It is easy to check that
\begin{equation}\label{proj:proj}
P_k=P_kP_{[k-1,k+1]}.
\end{equation}
Then we have
\begin{equation}\label{disp:estimate1}
\begin{split}
P_ke^{it\Lambda}f(x)&=(2\pi)^{-d}\int_{\R^d}K_k(t,x-y)P_kf(y)dy,\\
K_k(t,x)&:=\int_{\R^d}e^{i(x\cdot\xi+t\w{\xi})}\psi_{[k-1,k+1]}(\xi)d\xi.
\end{split}
\end{equation}
According to Corollary 2.36 and 2.38 in \cite{NS11book}, for any $t\ge1$, it holds that
\begin{equation*}
\|K_k(t,x)\|_{L^\infty(\R^d)}\ls2^{k(d/2+1)}t^{-d/2}.
\end{equation*}
This, together with Young's inequality and \eqref{disp:estimate1}, leads to
\begin{equation*}
\|P_ke^{it\Lambda}f\|_{L^\infty(\R^d)}\ls\|K_k\|_{L^\infty(\R^d)}\|P_kf\|_{L^1(\R^d)}
\ls2^{k(d/2+1)}t^{-d/2}\|P_kf\|_{L^1(\R^d)}.
\end{equation*}
The estimate of $\|P_ke^{-it\Lambda}f\|_{L^\infty(\R^d)}$ is analogous, we omit the details.
\end{proof}
Through minor modifications for the proof of Lemma 3.2 in \cite{Zheng22}, we next derive the following result.
\begin{lemma}[Linear Strichartz estimate]\label{lem:Stricha}
For any function $f$, integer $k\ge-1$ and $t\ge1$, it holds that
\begin{equation}\label{Strichartz}
\|P_ke^{\pm is\Lambda}f\|_{L^2([1,t])L^\infty(\R^d)}
\ls2^{kd/2}c_d(t)\|P_kf\|_{L^2(\R^d)},
\end{equation}
where $c_d(t)=1$ if $d\ge3$ and $c_2(t)=\ln^{1/2}t$.
Moreover, for $d=2$ and $p\in(2,\infty)$, one has
\begin{equation}\label{Strichartz'}
\|P_ke^{\pm is\Lambda}f\|_{L^p([1,t])L^\infty(\R^2)}
\ls\frac{2^k}{(p-2)^{1/p}}\|P_kf\|_{L^2(\R^2)}.
\end{equation}
\end{lemma}
\begin{remark}\label{Strichartz:d=1}
If $d=1$, set $c_1(t)=t^{1/4}$, then \eqref{Strichartz} still holds with $2^{kd/2}$ replaced by $2^{3k/4}$.
\end{remark}
\begin{proof}
For any $q\ge2$, write the operator
\begin{equation*}
T: f\mapsto P_{[k-1,k+1]}e^{\pm is\Lambda}f,
\quad L^2(\R^d)\rightarrow L^q([1,t])L^\infty(\R^d).
\end{equation*}
Then the adjoint operator of $T$ is
\begin{equation*}
T^*: g\mapsto \int_1^tP_{[k-1,k+1]}e^{\mp is\Lambda}g(s)ds, \quad L^{q'}([1,t])L^1(\R^d)\rightarrow L^2(\R^d),
\end{equation*}
where $q'=\frac{q}{q-1}$. Moreover,
\begin{equation}\label{Strichartz1}
\|T\|=\|T^*\|=\|TT^*\|^\frac12.
\end{equation}
In addition, one has
\begin{equation}\label{Strichartz2}
TT^*: g\mapsto \int_1^tP_{[k-1,k+1]}^2e^{\mp i(s'-s)\Lambda}g(s')ds',
\quad L^{q'}([1,t])L^1(\R^d)\rightarrow L^q([1,t])L^\infty(\R^d).
\end{equation}
It follows from \eqref{disp:estimate} and the Bernstein inequality that
\begin{equation}\label{Strichartz3}
\|P_{[k-1,k+1]}^2e^{\mp i(s'-s)\Lambda}g(s')\|_{L^\infty(\R^d)}
\ls2^{kd}(1+|s'-s|)^{-d/2}\|g(s')\|_{L^1(\R^d)}.
\end{equation}
Given an interval $I\subset\R$, denote the characteristic function
\begin{equation}\label{charact:def}
\id_{I}(t)=\left\{
\begin{aligned}
&1,\qquad t\in I,\\
&0,\qquad t\not\in I.
\end{aligned}
\right.
\end{equation}
Applying Young's inequality with \eqref{Strichartz2}--\eqref{Strichartz3} yields for $q=2$
\begin{equation}\label{Strichartz4}
\begin{split}
\|TT^*g\|_{L^2([1,t])L^\infty(\R^d)}&\ls2^{kd}
\Big\|\Big(\id_{[-t,t]}(\cdot)(1+|\cdot|)^{-d/2}\Big)
\ast\Big(\id_{[1,t]}(\cdot)\|g(\cdot)\|_{L^1(\R^d)}\Big)(s)\Big\|_{L^2([1,t])}\\
&\ls2^{kd}\|(1+|\cdot|)^{-d/2}\|_{L^1([-t,t])}\|g\|_{L^2([1,t])L^1(\R^d)}\\
&\ls2^{kd}c_d^2(t)\|g\|_{L^2([1,t])L^1(\R^d)}.
\end{split}
\end{equation}
Let $g=P_kf$ in \eqref{Strichartz4}. Then \eqref{Strichartz} is achieved from \eqref{Strichartz1} and \eqref{Strichartz4}.
Next, we turn to the proof of \eqref{Strichartz'}.
By using Young's inequality with $q=p>2$ for $TT^*g$, we can obtain
\begin{equation*}
\begin{split}
\|TT^*g\|_{L^p([1,t])L^\infty(\R^2)}&\ls2^{2k}
\|(1+|\cdot|)^{-1}\|_{L^{p/2}([-t,t])}\|g\|_{L^{p'}([1,t])L^1(\R^2)},\\
&\ls\frac{2^{2k}}{(p-2)^{2/p}}\|g\|_{L^{p'}([1,t])L^1(\R^2)},
\end{split}
\end{equation*}
which yields \eqref{Strichartz'}.
\end{proof}

\subsection{Paradifferential calculus}

As in Section 3 of \cite{Zheng22} or Section 3.2 of \cite{DIP17}, we collecting the following definitions.
\begin{definition}
Given a symbol $a=a(x,\zeta):\R^d\times(\R^d\setminus\{0\})\rightarrow\C$, define the Weyl quantization operator $T_a$ as
\begin{equation}\label{para:def}
\sF(T_af)(\xi):=C\int_{\R^d}\psi_{\le-10}\Big(\frac{|\xi-\eta|}{|\xi+\eta|}\Big)
(\sF_xa)(\xi-\eta,\frac{\xi+\eta}{2})\hat f(\eta)d\eta,
\end{equation}
where $\psi_{\le-10}(x)=\psi(2^{10}|x|)$ and $C$ is a normalization constant such that $T_1={\rm Id}$.
\end{definition}
\begin{remark}
When $\xi=\eta=0$, $T_a$ has no definition.
In fact, we will always deal with $P_{\ge0}T_af$ or $T_aP_{\ge0}f$, which means that
the situation of $\xi=\eta=0$ in \eqref{para:def} does not happen.
\end{remark}

\begin{lemma}\label{lem:para}
\begin{description}
  \item[(i)] If $a$ is real valued, then $T_a$ is self adjoint.
  \item[(ii)] If $a=a(\zeta)$, then $T_af=a(\frac{\p_x}{i})f$ is a Fourier multiplier.
\end{description}
\end{lemma}
\begin{proof}
It follows directly from the definition, we omit the proof here.
\end{proof}

\begin{definition}[Symbol norm]
For $p\in[1,\infty]$ and $m\in\R$, define
\begin{equation*}
\|a\|_{\sL_m^p}:=\sup_{\zeta\in\R^d}(1+|\zeta|)^{-m}\||a|(x,\zeta)\|_{L_x^p(\R^d)},
\quad |a|(x,\zeta):=\sum_{|\alpha|\le\mathfrak{c}_d}|\zeta|^{|\alpha|}
|D_\zeta^\alpha a(x,\zeta)|,
\end{equation*}
where $\mathfrak{c}_d$ is some integer depending on the space dimensions $d$.
\end{definition}

\begin{lemma}\label{lem:paraL2}
\begin{description}
  \item[(i)] For fixed $m,s\in\R$, we have $\|T_af\|_{H^s(\R^d)}\ls\|a\|_{\sL_m^\infty}\|f\|_{H^{s+m}(\R^d)}$.
  \item[(ii)] For fixed $0\le m<s$ with $m,s\in\R$, $\|H(f,g)\|_{H^s(\R^d)}\ls\|f\|_{W^{m,\infty}(\R^d)}\|g\|_{H^{s-m}(\R^d)}$ holds,
      where the remainder term
\begin{equation}\label{remainder:def}
H(f,g):=fg-T_fg-T_gf.
\end{equation}
\end{description}
\end{lemma}
\begin{proof}
The proofs see \cite[Lemma 3.11, 3.13]{Zheng22}.
\end{proof}

\begin{definition}
Given symbols $a_1,\cdots,a_n$, define the error operator
\begin{equation}\label{error:def}
E(a_1,\cdots,a_n):=T_{a_1}\cdots T_{a_n}-T_{a_1\cdots a_n}.
\end{equation}
\end{definition}

\begin{lemma}\label{lem:error}
For fixed $s,m_j\in\R$, we have
\begin{equation*}
\|E(a_1,\cdots,a_n)f\|_{H^s(\R^d)}\ls\prod_{j=1}^n(\|a_j\|_{\sL_{m_j}^\infty}
+\|\nabla_xa_j\|_{\sL_{m_j}^\infty})\|f\|_{H^{s+\sum_{j=1}^nm_j-1}(\R^d)}.
\end{equation*}
\end{lemma}
\begin{proof}
The proof sees \cite[Lemma 3.15]{Zheng22}.
\end{proof}

\section{Higher order energy estimate}

\subsection{Good unknown}

Without loss of generality, we assume that $F(u,\p u,\p^2u)$ in \eqref{KG} is independent of $\p_t^2u$
and is linear in $\p\p_xu$, which has the following form
\begin{equation}\label{nonlinear}
F(u,\p u,\p\p_xu)=2\sum_{j=1}^dQ^{0j}(u,\p u)\p^2_{tj}u
+\sum_{j,l=1}^dQ^{jl}(u,\p u)\p^2_{jl}u+S(u, \p u),
\end{equation}
where $Q^{0j}(0,0)=Q^{jl}(0,0)=0$, $S(u, \p u)$ is quadratic in $(u, \p u)$.

Let $u$ be the real-valued solution to \eqref{KG}. As in \cite{IP14}, set
\begin{equation}\label{U:def}
U_\pm:=(\p_t\pm i\Lambda)u,\quad U:=U_+.
\end{equation}
In addition, we introduce the good unknown
\begin{equation}\label{good:unknown}
\cU:=\p_tu-iT_{Q^{0j}\zeta_j}u+iT_{\sqrt{1+q}}\Lambda u,
\end{equation}
where $q(x,\zeta):=(Q^{jl}+Q^{0j}Q^{0l})\zeta_j\zeta_l\Lambda^{-2}(\zeta)$, the
summations $\ds\sum_{j=1}^d$ in  $T_{Q^{0j}\zeta_j}u$ of \eqref{good:unknown} and $\ds\sum_{j,l=1}^d$ in $q(x,\zeta)$ are ignored.

In this section, we are devoted to establishing the following higher order energy estimate.
\begin{proposition}\label{prop:high:energy}
Let $N$ be given in Theorem \ref{thm1} and $\|U\|_{H^N}$ be sufficiently small. Then it holds that
\begin{equation*}
\begin{split}
\|P_{\ge1}\Lambda^N\cU(t)\|^2_{L^2(\R^d)}&\ls\int_1^t\Big(\sum_{k\ge-1}
2^{k(2d+5+1/8)}\|P_kU(s)\|_{L^\infty(\R^d)}\Big)^2\|U(s)\|^2_{H^N(\R^d)}ds\\
&\quad+\|U(1)\|^2_{H^N(\R^d)}+\|U(t)\|^3_{H^N(\R^d)}.
\end{split}
\end{equation*}
\end{proposition}

In order to prove Proposition \ref{prop:high:energy}, we now give some auxiliary results.

\begin{lemma}\label{lem:goodunknown}
Under the assumptions of Proposition \ref{prop:high:energy}, we have
\begin{equation}\label{good:unknown1}
\begin{split}
\|P_{\ge0}(\cU-U)\|_{H^N(\R^d)}&\ls\|U\|_{W^{3,\infty}(\R^d)}\|U\|_{H^N(\R^d)}
\ls\|U\|^2_{H^N(\R^d)},\\
|q(x,\zeta)|&\le1/2.
\end{split}
\end{equation}
\end{lemma}
\begin{remark}
Thanks to $|q(x,\zeta)|\le1/2$ in \eqref{good:unknown1}, then $T_{\sqrt{1+q}}\Lambda u$ in \eqref{good:unknown}
is well defined.
\end{remark}

\begin{proof}
It follows from Lemma \ref{lem:para} (ii) and the definitions \eqref{U:def}-\eqref{good:unknown} that
\begin{equation*}
\cU-U=-iT_{Q^{0j}\zeta_j\Lambda^{-1}(\zeta)}\Lambda u
-iE(Q^{0j}\zeta_j,\Lambda^{-1}(\zeta))\Lambda u+iT_{\sqrt{1+q}-1}\Lambda u.
\end{equation*}
Applying Lemmas \ref{lem:paraL2} and \ref{lem:error} yields
\begin{equation*}
\begin{split}
\|P_{\ge0}(\cU-U)\|_{H^N(\R^d)}&\ls(\|Q^{0j}\|_{W^{1,\infty}(\R^d)}
+\|\sqrt{1+q}-1\|_{\sL_0^\infty})\|U\|_{H^N(\R^d)}\\
&\ls(\|Q^{j\alpha}\|_{W^{1,\infty}(\R^d)}+\|Q^{0j}\|^2_{L^\infty(\R^d)})\|U\|_{H^N(\R^d)}\\
&\ls\|U\|_{W^{3,\infty}(\R^d)}\|U\|_{H^N(\R^d)},
\end{split}
\end{equation*}
which derives the first line in \eqref{good:unknown1}. On the other hand,
the second line in \eqref{good:unknown1} is obvious by $Q^{0j}(0,0)=Q^{jl}(0,0)=0$
and the smallness of $(u,\p u)$.
\end{proof}

By definition \eqref{good:unknown}, the equation in \eqref{KG} is reduced to
\begin{equation}\label{cU:eqn}
(\p_t-iT_{Q^{0j}\zeta_j+\sqrt{1+q}\Lambda(\zeta)})\cU=\cS+\cQ+\cC,
\end{equation}
where
\begin{equation}\label{cU:eqn1}
\begin{split}
\cS&:=S(u,\p u)+2H(Q^{0j},\p^2_{tj}u)+H(Q^{jl},\p^2_{jl}u),\\
\cQ&:=2T_{\p^2_{tj}u}Q^{0j}+T_{\p^2_{jl}u}Q^{jl}
-iT_{\cF_1\zeta_j\Lambda^{-1}(\zeta)}\Lambda u+iT_{\cF_1}\Lambda u
+2iE(Q^{0j},\zeta_j)\p_tu\\
&\quad-E(Q^{jl},\zeta_j\zeta_l\Lambda^{-1}(\zeta))\Lambda u
-iE(\cF_1\zeta_j,\Lambda^{-1}(\zeta))\Lambda u+iE(q/2,\Lambda(\zeta))\p_tu\\
&\quad-E(\Lambda(\zeta),Q^{0j}\zeta_j,\Lambda^{-1}(\zeta))\Lambda u
+E(\Lambda(\zeta),q/2)\Lambda u,\\
\cC&:=-E(Q^{0j}\zeta_j,Q^{0l}\zeta_l,\Lambda^{-1}(\zeta))\Lambda u
+iE(\sqrt{1+q}-1-q/2,\Lambda(\zeta))\p_tu\\
&\quad+E(Q^{0j}\zeta_j,\sqrt{1+q}-1)\Lambda u
-E((\sqrt{1+q}-1)\Lambda(\zeta),Q^{0j}\zeta_j,\Lambda^{-1}(\zeta))\Lambda u\\
&\quad+E((\sqrt{1+q}-1)\Lambda(\zeta),\sqrt{1+q}-1)\Lambda u
+E(\Lambda(\zeta),\sqrt{1+q}-1-q/2)\Lambda u\\
&\quad-iT_{\cF_2\zeta_j\Lambda^{-1}(\zeta)}\Lambda u
-iE(\cF_2\zeta_j,\Lambda^{-1}(\zeta))\Lambda u
+iT_{((1+q)^{-1/2}-1)\cF_1+(1+q)^{-1/2}\cF_2}\Lambda u,
\end{split}
\end{equation}
and $\cF_1(0,0,0)=0$, $\cF_1=\cF_1(u,\p u,\p\p_xu)$ is linear in $(u,\p u,\p\p_xu)$, $\cF_2=\cF_2(u,\p u,\p\p_xu)$ is quadratic in $(u,\p u,\p\p_xu)$.
The proof of \eqref{cU:eqn} is put in Appendix \ref{section:B}.

By Lemma \ref{lem:para} (i) and \eqref{cU:eqn}, it is easy to get that $\w{T_{Q^{0j}\zeta_j+\sqrt{1+q}\Lambda(\zeta)}f,f}$ is real and
\begin{equation}\label{high:energy}
\begin{split}
\frac{d}{dt}\|P_{\ge1}\Lambda^N\cU\|^2_{L^2(\R^d)}
&=2\Re\w{(\p_t-iT_{Q^{0j}\zeta_j+\sqrt{1+q}\Lambda(\zeta)})
P_{\ge1}\Lambda^N\cU,P_{\ge1}\Lambda^N\cU}\\
&=2(E_\cS+E_\cQ+E_\cC),\\
\end{split}
\end{equation}
where $\w{f,g}:=\int_{\R^d}f(x)\overline{g(x)}dx$ and
\begin{equation}\label{nonlin:energey}
\begin{split}
E_\cS&=\Re\w{P_{\ge1}\Lambda^N\cS,P_{\ge1}\Lambda^NU},\\
E_\cQ&=\Re\w{[P_{\ge1}\Lambda^N,iT_{Q^{0j}\zeta_j+(\sqrt{1+q}-1)\Lambda(\zeta)}]\cU
+P_{\ge1}\Lambda^N\cQ,P_{\ge1}\Lambda^N\cU},\\
E_\cC&=\Re\w{P_{\ge1}\Lambda^N\cS,P_{\ge1}\Lambda^N(\cU-U)}
+\Re\w{P_{\ge1}\Lambda^N\cC,P_{\ge1}\Lambda^N\cU}.
\end{split}
\end{equation}

\begin{lemma}\label{lem:quart:energy}
Under the assumptions of Proposition \ref{prop:high:energy}, we then have
\begin{equation}\label{quart:energy}
|E_\cC|\ls\|U\|^2_{W^{3,\infty}(\R^d)}\|U\|^2_{H^N(\R^d)}.
\end{equation}
\end{lemma}
\begin{proof}
It follows from Lemmas \ref{lem:paraL2} and \ref{lem:error} that
\begin{equation*}
\begin{split}
\|P_{\ge1}\Lambda^N\cS\|_{L^2(\R^d)}&\ls\|\cS\|_{H^N(\R^d)}\ls\|\p u\|_{W^{1,\infty}(\R^d)}
\|U\|_{H^N(\R^d)}\ls\|U\|_{W^{3,\infty}(\R^d)}\|U\|_{H^N(\R^d)},\\
\|P_{\ge1}\Lambda^N\cC\|_{L^2(\R^d)}&\ls\|U\|^2_{W^{3,\infty}(\R^d)}\|U\|_{H^N(\R^d)}.
\end{split}
\end{equation*}
This, together with Lemma \ref{lem:goodunknown}, yields \eqref{quart:energy}.
\end{proof}

\subsection{Energy estimate I}
For $\mu,\nu=\pm$, define the phase function
\begin{equation}\label{phase:def}
\Phi_{\mu\nu}(\xi_1,\xi_2)
:=-\Lambda(\xi_1+\xi_2)+\mu\Lambda(\xi_1)+\nu\Lambda(\xi_2).
\end{equation}
The following lemma shows that the phase function $|\Phi_{\mu\nu}|$ has a lower bound.
\begin{lemma}\label{lem:phase}
For $l\ge1$, we have
\begin{equation}\label{phi:lowerbound}
|\Phi^{-1}_{\mu\nu}(\xi-\eta,\eta)|\ls1+\min\{|\xi|,|\eta|,|\xi-\eta|\},\quad
|\nabla_{\xi_1,\xi_2}^l\Phi_{\mu\nu}|\ls\min\{1,|\Phi_{\mu\nu}|\}
\end{equation}
and
\begin{equation}\label{phi:deri:bound}
|\nabla_{\xi_1,\xi_2}^l\Phi^{-1}_{\mu\nu}|\ls|\Phi^{-1}_{\mu\nu}|.
\end{equation}
\end{lemma}
\begin{proof}
\eqref{phi:lowerbound} comes from Lemma 5.1 of \cite{IP13}.
\eqref{phi:deri:bound} can be obtained by \eqref{phi:lowerbound} and Leibniz's rules.
\end{proof}
For a function $m(\xi_1,\xi_2):\R^d\times\R^d\rightarrow\C$, define the bilinear pseudoproduct operator
\begin{equation}\label{bilinear:def}
B_m(f,g):=\sF_\xi^{-1}\int_{\R^d}m(\xi-\eta,\eta)\hat f(\xi-\eta)\hat g(\eta)d\eta.
\end{equation}
Then \eqref{KG} can be reformulated to
\begin{equation}\label{U:eqn}
(\p_t-i\Lambda)U=\cN:=\sum_{\mu,\nu=\pm}B_{a_{\mu\nu}}(U_\mu,U_\nu)+\cC_1,
\end{equation}
where $\cN$ is real, $\cC_1$ is at least cubic in $U$ and $a_{\mu\nu}=a_{\mu\nu}(\xi-\eta,\eta)$ is a linear combination of the products of
\begin{equation}\label{symbol:a}
1,\eta_j,\frac{1}{\Lambda(\eta)},\frac{1}{\Lambda(\xi-\eta)},\frac{\eta_j\eta_l}{\Lambda(\eta)},
\frac{\xi_l-\eta_l}{\Lambda(\xi-\eta)},\quad j,l=1,\cdots,d.
\end{equation}
On the other hand, \eqref{U:eqn} can be rewritten as
\begin{equation}\label{U:eqn1}
\p_tU_\mu=i\mu\Lambda U_\mu+\cN_\mu,\quad\mu=\pm,\quad\cN_\pm:=\cN.
\end{equation}

\begin{lemma}\label{lem:sem:energy}
Under the assumptions of Proposition \ref{prop:high:energy}, we then have
\begin{equation}\label{sem:energy}
\begin{split}
\Big|\int_1^tE_\cS(s)ds\Big|&\ls\sum_{k\ge-1}2^{k(2d+5+1/8)}\int_1^t
\|P_kU(s)\|_{L^\infty(\R^d)}\|U(s)\|_{W^{1,\infty}(\R^d)}\|U(s)\|^2_{H^N(\R^d)}ds\\
&\quad+\|U(1)\|^2_{H^N(\R^d)}+\|U(t)\|^3_{H^N(\R^d)}.
\end{split}
\end{equation}
\end{lemma}
\begin{proof}
By \eqref{cU:eqn1} and \eqref{nonlin:energey}, it is easy to find that $E_\cS$ is a linear combination of such terms
\begin{equation*}
E_\cS^{\mu\nu}=\Re\w{P_{\ge1}\Lambda^{N+1}\cT_1H(\cT_2U_\mu,\cT_3U_\nu),
P_{\ge1}\Lambda^NU}
\end{equation*}
and
\begin{equation*}
E_{\cS_1}^{\mu\nu}=\Re\w{P_{\ge1}\Lambda^N\cT_1(\cT_2U_\mu\cT_3U_\nu),
P_{\ge1}\Lambda^NU},
\end{equation*}
where $\cT_1,\cT_2,\cT_3$ are the standard Calderon-Zygmund operators and $\mu,\nu=\pm$.

At first, we deal with $E_\cS^{\mu\nu}$. Set
\begin{equation}\label{sem:energy1}
\begin{split}
I_\cS^{\mu\nu}[f,g,h]&:=\Re\w{B_{m_\cS}(f,g),h}=\Re\iint_{(\R^d)^2}m_\cS(\xi_1,\xi_2)
\hat f(\xi_1)\hat g(\xi_2)\overline{\hat h(\xi_1+\xi_2)}d\xi_1d\xi_2,\\
m_\cS(\xi_1,\xi_2)&:=-iC\Phi^{-1}_{\mu\nu}(\xi_1,\xi_2)
\Big[1-\psi_{\le-10}\Big(\frac{|\xi_1|}{|\xi_1+2\xi_2|}\Big)
-\psi_{\le-10}\Big(\frac{|\xi_2|}{|2\xi_1+\xi_2|}\Big)\Big],\\
I_\cS^{\mu\nu}&:=I_\cS^{\mu\nu}[\cT_2U_\mu,\cT_3U_\nu,
P^2_{\ge1}\Lambda^{2N+1}\cT_1^*U].
\end{split}
\end{equation}
It follows from direct computation and \eqref{phase:def}, \eqref{U:eqn1} that
\begin{equation*}
\begin{split}
\frac{dI_\cS^{\mu\nu}}{dt}
&=I_\cS^{\mu\nu}[\cT_2\p_tU_\mu,\cT_3U_\nu,P^2_{\ge1}\Lambda^{2N+1}\cT_1^*U]
+I_\cS^{\mu\nu}[\cT_2U_\mu,\cT_3\p_tU_\nu,P^2_{\ge1}\Lambda^{2N+1}\cT_1^*U]\\
&\quad+I_\cS^{\mu\nu}[\cT_2U_\mu,\cT_3U_\nu,P^2_{\ge1}\Lambda^{2N+1}\cT_1^*\p_tU]\\
&=E_\cS^{\mu\nu}+I_\cS^{\mu\nu}[\cT_2\cN_\mu,\cT_3U_\nu,P^2_{\ge1}\Lambda^{2N+1}\cT_1^*U]
+I_\cS^{\mu\nu}[\cT_2U_\mu,\cT_3\cN_\nu,P^2_{\ge1}\Lambda^{2N+1}\cT_1^*U]\\
&\quad+I_\cS^{\mu\nu}[\cT_2U_\mu,\cT_3U_\nu,P^2_{\ge1}\Lambda^{2N+1}\cT_1^*\cN],
\end{split}
\end{equation*}
which yields
\begin{equation}\label{sem:energy2}
\begin{split}
\int_1^tE_\cS^{\mu\nu}(s)ds&=I_\cS^{\mu\nu}(t)-I_\cS^{\mu\nu}(1)
-\int_1^tI_\cS^{\mu\nu}[\cT_2\cN_\mu,\cT_3U_\nu,P^2_{\ge1}\Lambda^{2N+1}\cT_1^*U]ds\\
&\quad-\int_1^tI_\cS^{\mu\nu}[\cT_2U_\mu,\cT_3\cN_\nu,P^2_{\ge1}\Lambda^{2N+1}\cT_1^*U]ds\\
&\quad-\int_1^tI_\cS^{\mu\nu}[\cT_2U_\mu,\cT_3U_\nu,P^2_{\ge1}\Lambda^{2N+1}\cT_1^*\cN]ds.
\end{split}
\end{equation}
Due to the support property of $\psi_{\le-10}$ in \eqref{sem:energy1}, one can see that
\begin{equation*}
I_\cS^{\mu\nu}[\cT_2f,\cT_3g,P^2_{\ge1}\Lambda^{2N+1}\cT_1^*h]
=\sum_{\substack{k\ge-1,\\k_1,k_2>k-20}}
I_\cS^{\mu\nu}[\cT_2P_{k_1}f,\cT_3P_{k_2}g,P^2_{\ge1}\Lambda^{2N+1}\cT_1^*P_kh].
\end{equation*}
By \eqref{bilinear:a}, we obtain that
\begin{equation*}
|I_\cS^{\mu\nu}[\cT_2P_{k_1}f,\cT_3P_{k_2}g,P^2_{\ge1}\Lambda^{2N+1}\cT_1^*P_kh]|
\ls 2^{k_1(2d+3)+k(2N+1)}\|P_{k_1}f\|_{L_x^\infty}\|P_{k_2}g\|_{L_x^2}\|P_kh\|_{L_x^2},
\end{equation*}
which implies for $l=0,1$
\begin{equation}\label{sem:energy3}
|I_\cS^{\mu\nu}[\cT_2f,\cT_3g,P^2_{\ge1}\Lambda^{2N+1}\cT_1^*h]|\ls
\sum_{k_1\ge-1}2^{k_1(2d+4+1/9+l)}\|P_{k_1}f\|_{L_x^\infty}\|g\|_{H_x^N}
\|h\|_{H_x^{N-l}}.
\end{equation}
Analogously, we arrive at
\begin{equation*}
|I_\cS^{\mu\nu}[\cT_2f,\cT_3g,P^2_{\ge1}\Lambda^{2N+1}\cT_1^*h]|\ls
\sum_{k_2\ge-1}2^{k_2(2d+4+1/9+l)}\|P_{k_2}g\|_{L_x^\infty}\|f\|_{H_x^N}
\|h\|_{H_x^{N-l}}.
\end{equation*}
Choosing $l=0$ in \eqref{sem:energy3} yields
\begin{equation}\label{sem:energy4}
|I_\cS^{\mu\nu}(t)|\ls\sum_{k\ge-1}2^{2d+4+1/9}\|P_kU(t)\|_{L^\infty(\R^d)}
\|U(t)\|^2_{H^N(\R^d)}\ls\|U(t)\|^3_{H^N(\R^d)},
\end{equation}
where  $N>2d+4+1/9+d/2$ is used.

Denote
\begin{equation}\label{prod:frequency}
\begin{split}
\cX_k&=\cX_k^1\cup\cX_k^2,\\
\cX_k^1&=\{(k_1,k_2)\in\Z^2: |\max\{k_1,k_2\}-k|\le8,k_1,k_2\ge-1\},\\
\cX_k^2&=\{(k_1,k_2)\in\Z^2: \max\{k_1,k_2\}\ge k+8,|k_1-k_2|\le8,k_1,k_2\ge-1\}.
\end{split}
\end{equation}
As in \cite[page 784]{IP13}, if $P_k(P_{k_1}fP_{k_2}g)\neq0$, then $(k_1,k_2)\in\cX_k$.
Applying \eqref{bilinear:b} to $\cN$ in \eqref{U:eqn} derives
\begin{equation}\label{sem:energy5}
\begin{split}
&\quad\;|I_\cS^{\mu\nu}[\cT_2\cN_\mu,\cT_3U_\nu,P^2_{\ge1}\Lambda^{2N+1}\cT_1^*U]|\\
&\ls\sum_{k\ge-1}2^{k(2d+4+1/9)}\|P_k\cN\|_{L^\infty(\R^d)}\|U\|^2_{H^N(\R^d)}\\
&\ls\sum_{k\ge-1}\sum_{(k_1,k_2)\in\cX_k}2^{\max\{k_1,k_2\}(2d+5+1/9)}
\|P_{k_1}U\|_{L^\infty(\R^d)}\|P_{k_2}U\|_{L^\infty(\R^d)}\|U\|^2_{H^N(\R^d)}\\
&\ls\sum_{k\ge-1}2^{k(2d+5+1/8)}\|P_kU\|_{L^\infty(\R^d)}\|U\|_{W^{1,\infty}(\R^d)}
\|U\|^2_{H^N(\R^d)},
\end{split}
\end{equation}
where we have ignored the cubic term $\cC_1$ since it can be treated more easily.

Analogously,
\begin{equation}\label{sem:energy6}
\begin{split}
|I_\cS^{\mu\nu}[\cT_2U_\mu,\cT_3\cN_\nu,P^2_{\ge1}\Lambda^{2N+1}\cT_1^*U]|\ls
\sum_{k\ge-1}2^{k(2d+5+1/8)}\|P_kU\|_{L_x^\infty}\|U\|_{W_x^{1,\infty}}
\|U\|^2_{H_x^N}.
\end{split}
\end{equation}
Before taking the estimate on the last line in \eqref{sem:energy2}, we firstly treat $\|\cN\|_{H^{N-1}}$.
By using \eqref{bilinear:b} again, one has
\begin{equation*}
\|\cN\|_{H_x^{N-1}}\ls\Big\|\sum_{(k_1,k_2)\in\cX_k}2^{k(N-1)+k_2}
\|P_{k_1}U\|_{L_x^\infty}\|P_{k_2}U\|_{L_x^2}\Big\|_{\ell^2_k}
\ls\|U\|_{W_x^{1,\infty}}\|U\|_{H_x^N},
\end{equation*}
where $\ds\|A_k\|_{\ell^p_k}=(\sum_{k\ge-1}A_k^p)^{1/p},p\ge1$.
In addition, choosing $l=1$ in \eqref{sem:energy3} yields
\begin{equation}\label{sem:energy7}
\begin{split}
|I_\cS^{\mu\nu}[\cT_2U_\mu,\cT_3U_\nu,P^2_{\ge1}\Lambda^{2N+1}\cT_1^*\cN]|
&\ls\sum_{k\ge-1}2^{k(2d+5+1/8)}\|P_kU\|_{L_x^\infty}\|U\|_{H_x^N}\|\cN\|_{H_x^{N-1}}\\
&\ls\sum_{k\ge-1}2^{k(2d+5+1/8)}\|P_kU\|_{L_x^\infty}\|U\|_{W_x^{1,\infty}}
\|U\|^2_{H_x^N}.
\end{split}
\end{equation}
Next, we turn to the estimate of $E_{\cS_1}^{\mu\nu}$.
Similarly to $E_\cS^{\mu\nu}$, let
\begin{equation}\label{sem:energy8}
\begin{split}
I_{\cS_1}^{\mu\nu}[f,g,h]&:=\Re\iint_{(\R^d)^2}m_{\cS_1}(\xi-\eta,\eta)
\hat f(\xi-\eta)\hat g(\eta)\overline{\hat h(\xi)}d\xi d\eta,\\
m_{\cS_1}(\xi_1,\xi_2)&:=-i\Phi^{-1}_{\mu\nu}(\xi_1,\xi_2),\\
I_{\cS_1}^{\mu\nu}&:=I_{\cS_1}^{\mu\nu}[\cT_2U_\mu,\cT_3U_\nu,
P^2_{\ge1}\Lambda^{2N}\cT_1^*U].
\end{split}
\end{equation}
Then we arrive at
\begin{equation*}
\begin{split}
\int_1^tE_{\cS_1}^{\mu\nu}(s)ds&=I_{\cS_1}^{\mu\nu}(t)-I_{\cS_1}^{\mu\nu}(1)
-\int_1^tI_{\cS_1}^{\mu\nu}[\cT_2\cN_\mu,\cT_3U_\nu,P^2_{\ge1}\Lambda^{2N}\cT_1^*U]ds\\
&\quad-\int_1^tI_{\cS_1}^{\mu\nu}[\cT_2U_\mu,\cT_3\cN_\nu,P^2_{\ge1}\Lambda^{2N}\cT_1^*U]ds\\
&\quad-\int_1^tI_{\cS_1}^{\mu\nu}[\cT_2U_\mu,\cT_3U_\nu,P^2_{\ge1}\Lambda^{2N}\cT_1^*\cN]ds,
\end{split}
\end{equation*}
where
\begin{equation*}
I_{\cS_1}^{\mu\nu}[\cT_2f,\cT_3g,P^2_{\ge1}\Lambda^{2N}\cT_1^*h]
=\sum_{k\ge-1}\sum_{(k_1,k_2)\in\cX_k}
I_{\cS_1}^{\mu\nu}[\cT_2P_{k_1}f,\cT_3P_{k_2}g,P^2_{\ge1}\Lambda^{2N}\cT_1^*P_kh].
\end{equation*}
Note that the estimate of $I_{\cS_1}^{\mu\nu}$ is much easier  to be treated than $I_\cS^{\mu\nu}$, we omit it here.
Substituting \eqref{sem:energy4}--\eqref{sem:energy7} into \eqref{sem:energy2} derives \eqref{sem:energy}.
\end{proof}

\subsection{Energy estimate II}
\begin{lemma}\label{lem:qua:energy}
Under the assumptions of Proposition \ref{prop:high:energy}, we then have
\begin{equation}\label{qua:energy}
\begin{split}
\Big|\int_1^tE_\cQ(s)ds\Big|&\ls\int_1^t\Big(\sum_{k\ge-1}
2^{k(2d+5+1/8)}\|P_kU(s)\|_{L^\infty(\R^d)}\Big)^2\|U(s)\|^2_{H^N(\R^d)}ds\\
&\quad+\|U(1)\|^2_{H^N(\R^d)}+\|U(t)\|^3_{H^N(\R^d)}.
\end{split}
\end{equation}
\end{lemma}
\begin{proof}
Note that
\begin{equation*}
\begin{split}
[P_{\ge1}\Lambda^N,iT_{Q^{0j}\zeta_j+(\sqrt{1+q}-1)\Lambda(\zeta)}]
&=iE(\psi_{\ge1}(\zeta)\Lambda^N(\zeta),Q^{0j}\zeta_j+(\sqrt{1+q}-1)\Lambda(\zeta))\\
&\quad-iE(Q^{0j}\zeta_j+(\sqrt{1+q}-1)\Lambda(\zeta),\psi_{\ge1}(\zeta)\Lambda^N(\zeta)).
\end{split}
\end{equation*}
Analogously to $E_\cC$ in \eqref{nonlin:energey}, $E_\cQ$ is a linear combination of
\begin{equation}\label{qua:energy1}
\begin{split}
E_\cQ^{\mu\nu}&:=\Re\w{B_{m_\cQ}(U_\mu,\cU_\nu),\cU}
=\Re\iint_{(\R^d)^2}m_\cQ(\xi_1,\xi_2)
\hat U_\mu(\xi_1)\hat\cU_\nu(\xi_2)\overline{\hat\cU(\xi_1+\xi_2)}d\xi_1d\xi_2\\
\end{split}
\end{equation}
with
\begin{equation*}
\begin{split}
m_\cQ(\xi_1,\xi_2)&:=C\psi_{\le-10}\Big(\frac{|\xi_1|}{|\xi_1+2\xi_2|}\Big)
n_1(\xi_1)n_2(\xi_2)n_3(\xi_1+\xi_2)\\
&\qquad\times\Big[n_4(\xi_1+\xi_2)n_5(\xi_2)
-n_4(\frac{\xi_1+2\xi_2}{2})n_5(\frac{\xi_1+2\xi_2}{2})\Big],
\end{split}
\end{equation*}
where $n_l\in S^{m_l}_{1,0}$ (H\"ormander class), $l=1,\cdots,5$, $\sum_{l=1}^5m_l=2N+1$, and $n_2,n_3=0$ on $\supp\psi_{-1}$.
Denote
\begin{equation*}
I_\cQ^{\mu\nu}[f,g,h]:=\Re\w{B_{-i\Phi^{-1}_{\mu\nu}m_\cQ}(f,g),h},
\qquad I_\cQ^{\mu\nu}:=I_\cQ^{\mu\nu}[U_\mu,\cU_\nu,\cU].
\end{equation*}
As in Lemma \ref{lem:sem:energy}, we can obtain
\begin{equation}\label{qua:energy2}
\begin{split}
\int_1^tE_\cQ(s)ds&=I_\cQ^{\mu\nu}(t)-I_\cQ^{\mu\nu}(1)
-\int_1^tI_\cQ^{\mu\nu}[\cN_\mu,\cU_\nu,\cU]ds\\
&\quad-\int_1^tI_\cQ^{\mu\nu}[U_\mu,(\p_t-i\nu\Lambda)\cU_\nu,\cU]ds
-\int_1^tI_\cQ^{\mu\nu}[U_\mu,U_\nu,(\p_t-i\Lambda)\cU]ds.
\end{split}
\end{equation}
For the term $I_\cQ^{\mu\nu}(t)$ in \eqref{qua:energy2}, it can be deduced from \eqref{prod:frequency} and \eqref{bilinear:c} that
\begin{equation*}
\begin{split}
|I_\cQ^{\mu\nu}(t)|&\ls\sum_{\substack{k\ge-1,\\(k_1,k_2)\in\cX_k}}
|I_\cQ^{\mu\nu}[P_{k_1}U_\mu,P_{k_2}\cU_\nu,P_k\cU]|\\
&\ls\sum_{\substack{k\ge-1,(k_1,k_2)\in\cX_k^1,\\|k-k_2|\le O(1)}}2^{k_1(2d+4)+2Nk}
\|P_{k_1}U\|_{L^\infty(\R^d)}\|P_{\ge0}P_{k_2}\cU\|_{L^2(\R^d)}\|P_{\ge0}P_k\cU\|_{L^2(\R^d)}\\
&\quad+\sum_{k\ge-1}2^{k(2d+4+1/9)}\|P_kU\|_{L^\infty(\R^d)}\|P_{\ge0}\cU\|^2_{H^N(\R^d)},
\end{split}
\end{equation*}
where the last line above for the case $(k_1,k_2)\in\cX_k^2$ can be treated as in \eqref{sem:energy4}.
Together with Lemma \ref{lem:goodunknown}, one achieves
\begin{equation}\label{qua:energy3}
|I_\cQ^{\mu\nu}(t)|\ls\|U(t)\|^3_{H^N(\R^d)}.
\end{equation}
Analogously,
\begin{equation}\label{qua:energy4}
\begin{split}
|I_\cQ^{\mu\nu}[\cN_\mu,\cU_\nu,\cU]|
&\ls\sum_{k\ge-1}2^{k(2d+4+1/9)}\|P_k\cN\|_{L^\infty(\R^d)}\|P_{\ge0}\cU\|^2_{H^N(\R^d)}\\
&\ls\sum_{k\ge-1}2^{k(2d+5+1/8)}\|P_kU\|_{L^\infty(\R^d)}\|U\|_{W^{3,\infty}(\R^d)}
\|U\|^2_{H^N(\R^d)}.
\end{split}
\end{equation}
For the second line in \eqref{qua:energy2}, it follows from \eqref{cU:eqn} that
\begin{equation}\label{qua:energy5}
\begin{split}
&\quad\; I_\cQ^{\mu\nu}[U_\mu,(\p_t-i\nu\Lambda)\cU_\nu,\cU]
+I_\cQ^{\mu\nu}[U_\mu,U_\nu,(\p_t-i\Lambda)\cU]\\
&=I_\cQ^{\mu\nu}[U_\mu,i\nu T_{Q^{0j}\zeta_j+(\sqrt{1+q}-1)\Lambda(\zeta)}\cU_\nu,
\cU]+I_\cQ^{\mu\nu}[U_\mu,\cU_\nu,iT_{Q^{0j}\zeta_j+(\sqrt{1+q}-1)\Lambda(\zeta)}\cU]\\
&\quad+I_\cQ^{\mu\nu}[U_\mu,(\cS+\cQ+\cC)_\nu,\cU]
+I_\cQ^{\mu\nu}[U_\mu,\cU_\nu,\cS+\cQ+\cC],
\end{split}
\end{equation}
where $(\cS+\cQ+\cC)_+=\cS+\cQ+\cC$ and $(\cS+\cQ+\cC)_-=\overline{\cS+\cQ+\cC}$.
By using \eqref{bilinear:c} again, we have
\begin{equation}\label{qua:energy6}
\begin{split}
&|I_\cQ^{\mu\nu}[U_\mu,(\cS+\cQ+\cC)_\nu,\cU]|
+|I_\cQ^{\mu\nu}[U_\mu,\cU_\nu,\cS+\cQ+\cC]|\\
\ls&\sum_{k\ge-1}2^{k(2d+5+1/8)}\|P_kU\|_{L^\infty(\R^d)}\|U\|_{W^{3,\infty}(\R^d)}
\|U\|^2_{H^N(\R^d)}.
\end{split}
\end{equation}
At last, we turn to the estimate of the second line in \eqref{qua:energy5}. Denote $q_1(x,\zeta):=Q^{0j}\zeta_j+(\sqrt{1+q}-1)\Lambda(\zeta)$.
For $\nu=-$, it can be derived from \eqref{bilinear:d} that
\begin{equation}\label{qua:energy7}
|I_\cQ^{\mu-}[U_\mu,-iT_{q_1}\cU_-,\cU]+I_\cQ^{\mu-}[U_\mu,\cU_-,iT_{q_1}\cU]|
\ls\|U\|^2_{W^{3,\infty}(\R^d)}\|U\|^2_{H^N(\R^d)}.
\end{equation}
Next we deal with the case of $\nu=+$. Since $q_1(x,\zeta)$ is real, Lemma \ref{lem:para} (i) ensures that $T_{q_1}$ is self adjoint.
According to the definitions \eqref{para:def} and \eqref{bilinear:def}, we arrive at
\begin{equation}\label{qua:energy8}
\begin{split}
&\quad\; I_\cQ^{\mu+}[U_\mu,iT_{q_1}\cU,\cU]+I_\cQ^{\mu+}[U_\mu,\cU,iT_{q_1}\cU]\\
&=\Re\w{B_{\Phi^{-1}_{\mu+}m_\cQ}(U_\mu,T_{q_1}\cU),\cU}
-\Re\w{B_{\Phi^{-1}_{\mu+}m_\cQ}(U_\mu,\cU),T_{q_1}\cU}\\
&=\Re\w{B_{\Phi^{-1}_{\mu+}m_\cQ}(U_\mu,T_{q_1}\cU)
-T_{q_1}B_{\Phi^{-1}_{\mu+}m_\cQ}(U_\mu,\cU),\cU}\\
&=C\Re\iiint_{(\R^d)^3}\hat U_\mu(\xi_1)\hat\cU(\eta)\overline{\hat\cU(\xi_1+\xi_2)}
\Big[(\Phi^{-1}_{\mu+}m_\cQ)(\xi_1,\xi_2)\hat{q_1}(\xi_2-\eta,\frac{\xi_2+\eta}{2})
\psi_{\le-10}\Big(\frac{|\xi_2-\eta|}{|\xi_2+\eta|}\Big)\\
&\quad-(\Phi^{-1}_{\mu+}m_\cQ)(\xi_1,\eta)
\hat{q_1}(\xi_2-\eta,\frac{2\xi_1+\xi_2+\eta}{2})
\psi_{\le-10}\Big(\frac{|\xi_2-\eta|}{|2\xi_1+\xi_2+\eta|}\Big)\Big]d\xi_1 d\xi_2d\eta.
\end{split}
\end{equation}
Set
\begin{equation}\label{qua:energy9}
\begin{split}
&\quad\;|I_\cQ^{\mu+}[U_\mu,iT_{q_1}\cU,\cU]+I_\cQ^{\mu+}[U_\mu,\cU,iT_{q_1}\cU]|\\
&=\sum_{k,k_1,k_2,k_3\ge-1}|I_\cQ^{\mu+}[P_{k_1}U_\mu,iT_{P_{k_2}q_1}P_{k_3}\cU,P_k\cU]
+I_\cQ^{\mu+}[P_{k_1}U_\mu,P_{k_3}\cU,iT_{P_{k_2}q_1}P_k\cU]|\\
&:=\sum_{k,k_1,k_2,k_3\ge-1}|\w{\sT_{m_{\cQ_1}}(P_{k_1}U_\mu,P_{k_2}q_1,P_{k_3}\cU),P_k\cU}|,
\end{split}
\end{equation}
where $\sT_{m_{\cQ_1}}(f,g,h)$ is the trilinear pseudoproduct operator similarly defined in \eqref{bilinear:def}
\begin{equation}\label{trilinear:def}
\sT_m(f,g,h):=\sF_\xi^{-1}\iint_{(\R^d)^2}m(\xi-\eta,\eta-\zeta,\zeta)
\hat f(\xi-\eta)\hat g(\eta-\zeta)\hat h(\zeta)d\eta d\zeta.
\end{equation}
Denote the Schwarz kernel of $m_{\cQ_1}$ by $\cK(x,y,z)$.  Similarly to \eqref{bilinear1}, we can get
\begin{equation}\label{qua:energy10}
\|\cK(x,y,z)\|_{L^1((\R^d)^3)}\ls\sum_{l=0}^{[3d/2]+1}\sum_{n=1}^3
2^{lk_n}\|\psi_{k_n}(\xi_n)\p_{\xi_n}^lm_{\cQ_1}(\xi_1,\xi_2,\xi_3)\|_{L^\infty}.
\end{equation}
From \eqref{qua:energy8}, one can see that $|\xi_1|\approx2^{k_1}$, $|\xi_2-\eta|\approx2^{k_2}$, $|\eta|\approx2^{k_3}$, $|\xi_1+\xi_2|\approx2^k$, $k_1,k_2\le k_3-6$ and $|k_3-k|\le O(1)$.
In view of \eqref{phi:deri:bound}, it is required to control $\p_{\xi_3}^lm_{\cQ_1}$ in \eqref{qua:energy10} which is defined in
\eqref{qua:energy8} and \eqref{qua:energy9}.
Note that for $\xi=\xi_2$ or $\xi=\eta$
\begin{equation*}
\p_{\xi_2}\Phi^{-1}_{\mu+}(\xi_1,\xi)
=\Phi^{-2}_{\mu+}(\xi_1,\xi)(\nabla\Lambda(\xi_1+\xi)-\nabla\Lambda(\xi))
=\Phi^{-2}_{\mu+}(\xi_1,\xi)\int_0^1\xi_1\nabla^2\Lambda(\theta\xi_1+\xi)d\theta.
\end{equation*}
This, together with \eqref{bilinear2}-\eqref{bilinear3} yields
\begin{equation*}
\|\cK(x,y,z)\|_{L^1((\R^d)^3)}\ls2^{2Nk_3+([3d/2]+3)(k_1+k_2)}.
\end{equation*}
Therefore,
\begin{equation}\label{qua:energy11}
|I_\cQ^{\mu+}[U_\mu,iT_{q_1}\cU,\cU]+I_\cQ^{\mu+}[U_\mu,\cU,iT_{q_1}\cU]|\ls
\Big(\sum_{k\ge-1}2^{k([3d/2]+3)}\|P_kU\|_{L^\infty(\R^d)}\Big)^2\|U\|^2_{H^N(\R^d)}.
\end{equation}
Collecting \eqref{qua:energy2}--\eqref{qua:energy11} derives \eqref{qua:energy}.
\end{proof}

\begin{proof}[Proof of Proposition \ref{prop:high:energy}]
Substituting Lemma \ref{lem:quart:energy}, \ref{lem:sem:energy}, \ref{lem:qua:energy} into \eqref{high:energy}
implies Proposition \ref{prop:high:energy}.
\end{proof}

\section{Lower order energy estimate and proof of Theorem \ref{thm1}}
\subsection{Lower order energy estimate}
Define
\begin{equation}\label{profile:def}
\quad V:=V_+=e^{-it\Lambda}U,\quad V_-:=\overline{V}.
\end{equation}
Then \eqref{U:eqn} is reformulated into
\begin{equation}\label{profile:eqn}
\begin{split}
\hat V(t,\xi)&=\sum_{\mu,\nu=\pm}\int_1^t\int_{\R^d}e^{is\Phi_{\mu\nu}(\xi-\eta,\eta)}
a_{\mu\nu}(\xi-\eta,\eta)\hat V_\mu(s,\xi-\eta)\hat V_\nu(s,\eta)d\eta ds\\
&\quad+\hat V(1,\xi)+\int_1^te^{-is\Lambda(\xi)}\hat\cC_1(t,\xi)ds,
\end{split}
\end{equation}
where $\Phi_{\mu\nu}$ and $a_{\mu\nu}$ are defined by \eqref{phase:def} and \eqref{symbol:a}, respectively.
Thanks to \eqref{phi:lowerbound}, we can integrate the terms of \eqref{profile:eqn} by parts in $s$ to get
\begin{equation}\label{profile:eqn1}
\begin{split}
\hat V(t,\xi)&=\hat V(1,\xi)+\int_1^te^{-is\Lambda(\xi)}\hat\cC_1(t,\xi)ds
-i\sum_{\mu,\nu=\pm}\sF(e^{-is\Lambda}B_{\Phi^{-1}_{\mu\nu}a_{\mu\nu}}
(U_\mu,U_\nu))(s,\xi)\Big|_{s=1}^t\\
&\quad+i\sum_{\mu,\nu=\pm}\int_1^t\int_{\R^d}e^{is\Phi_{\mu\nu}(\xi-\eta,\eta)}
(\Phi^{-1}_{\mu\nu}a_{\mu\nu})(\xi-\eta,\eta)\p_t(\hat V_\mu(s,\xi-\eta)
\hat V_\nu(s,\eta))d\eta ds.
\end{split}
\end{equation}
Returning to the physical space $(t,x)$, one has from \eqref{profile:eqn1} that
\begin{equation}\label{profile:eqn2}
\begin{split}
V(t,x)&=V(1,x)+\int_1^te^{-is\Lambda}\cC_1ds-i\sum_{\mu,\nu=\pm}e^{-is\Lambda}
B_{\Phi^{-1}_{\mu\nu}a_{\mu\nu}}(U_\mu,U_\nu)(s,x)\Big|_{s=1}^t\\
&+i\sum_{\mu,\nu=\pm}\int_1^te^{-is\Lambda}\{
B_{\Phi^{-1}_{\mu\nu}a_{\mu\nu}}(\cN_\mu,U_\nu)
+B_{\Phi^{-1}_{\mu\nu}a_{\mu\nu}}(U_\mu,\cN_\nu)\}ds,
\end{split}
\end{equation}
where \eqref{U:eqn1} is  used.
\begin{lemma}\label{lem:low:energy}
Let $N$ be given in Theorem \ref{thm1} and $\|U\|_{H^N}$ be sufficiently small, then we have
\begin{equation}\label{low:energy}
\begin{split}
\|U(t)\|_{L^2(\R^d)}&\ls\|U(1)\|_{L^2(\R^d)}+\|U(1)\|^2_{H^N(\R^d)}+\|U(t)\|^2_{H^N(\R^d)}\\
&\quad+\int_1^t\|U(s)\|^2_{W^{2d+4,\infty}(\R^d)}\|U(s)\|_{H^N(\R^d)}ds.
\end{split}
\end{equation}
\end{lemma}
\begin{proof}
It follows from \eqref{profile:eqn2} that
\begin{equation}\label{low:energy1}
\begin{split}
\|P_k(V(t)-V(1))\|_{L^2(\R^d)}&\ls\sum_{(k_1,k_2)\in\cX_k}\Big(J_{kk_1k_2}^0(1)
+J_{kk_1k_2}^0(t)+\int_1^t(J_{kk_1k_2}^1(s)+J_{kk_1k_2}^2(s))ds\Big)\\
&\qquad+\int_1^t\|P_ke^{-is\Lambda}\cC_1\|_{L^2(\R^d)}ds,
\end{split}
\end{equation}
where
\begin{equation}\label{low:energy2}
\begin{split}
J_{kk_1k_2}^0(t)&:=\sum_{\mu,\nu=\pm}\|P_kB_{\Phi^{-1}_{\mu\nu}a_{\mu\nu}}
(P_{k_1}U_\mu,P_{k_2}U_\nu))(t)\|_{L^2(\R^d)},\\
J_{kk_1k_2}^1(s)&:=\sum_{\mu,\nu=\pm}\|P_kB_{\Phi^{-1}_{\mu\nu}a_{\mu\nu}}
(P_{k_1}\cN_\mu,P_{k_2}U_\nu)(s)\|_{L^2(\R^d)},\\
J_{kk_1k_2}^2(s)&:=\sum_{\mu,\nu=\pm}\|P_kB_{\Phi^{-1}_{\mu\nu}a_{\mu\nu}}
(P_{k_1}U_\mu,P_{k_2}\cN_\nu)(s)\|_{L^2(\R^d)}.
\end{split}
\end{equation}
We only deal with the case of $k_1\le k_2$ in $\cX_k$ of \eqref{low:energy2} since
the case of $k_1\ge k_2$ can be analogously treated.

\vskip 0.2 true cm

\noindent\textbf{Estimate of $J_{kk_1k_2}^0(t)$:}
Applying \eqref{bilinear:e} and the Bernstein inequality to obtain
\begin{equation}\label{scatter:pf}
\begin{split}
J_{kk_1k_2}^0(t)
&\ls2^{k_1(2d+3)+k_2}\|P_{k_1}U(t)\|_{L^\infty(\R^d)}\|P_{k_2}U(t)\|_{L^2(\R^d)}\\
&\ls2^{k_1(5d/2+3)+k_2}\|P_{k_1}U(t)\|_{L^2(\R^d)}\|P_{k_2}U(t)\|_{L^2(\R^d)}.
\end{split}
\end{equation}
Thus,
\begin{equation}\label{low:energy3}
\begin{split}
\Big\|\sum_{(k_1,k_2)\in\cX_k}J_{kk_1k_2}^0(t)\Big\|_{\ell^2_k}
\ls\|U(t)\|_{H^{[5d/2]+4}(\R^d)}\|U(t)\|_{H^2(\R^d)}\ls\|U(t)\|^2_{H^N(\R^d)},
\end{split}
\end{equation}
where $N>[5d/2]+4$ is used.

\vskip 0.2 true cm

\noindent\textbf{Estimate of $J_{kk_1k_2}^1(s)$:}
By \eqref{bilinear:b} and \eqref{bilinear:e}, one can arrive at
\begin{equation*}
\begin{split}
J_{kk_1k_2}^1(s)
&\ls2^{k_1(2d+3)+k_2}\|P_{k_1}\cN(s)\|_{L^\infty(\R^d)}\|P_{k_2}U(s)\|_{L^2(\R^d)}\\
&\ls2^{k_1(2d+3)-k_2}\sum_{(k_3,k_4)\in\cX_{k_1}}\sum_{\mu,\nu=\pm}
\|P_{k_1}B_{a_{\mu\nu}}(P_{k_3}U_\mu,P_{k_4}U_\nu)(s)\|_{L^\infty(\R^d)}\|U(s)\|_{H^3(\R^d)}\\
&\ls2^{-k_2}\sum_{k_3}2^{k_3(2d+3)}\|P_{k_3}U(s)\|_{L^\infty(\R^d)}
\|U(s)\|_{W^{2,\infty}(\R^d)}\|U(s)\|_{H^3(\R^d)}.
\end{split}
\end{equation*}
Similarly to \eqref{low:energy3}, we achieve
\begin{equation}\label{low:energy4}
\Big\|\sum_{(k_1,k_2)\in\cX_k}J_{kk_1k_2}^1(s)\Big\|_{\ell^2_k}
\ls\|U(s)\|_{W^{2d+4,\infty}(\R^d)}\|U(s)\|_{W^{2,\infty}(\R^d)}\|U(s)\|_{H^N(\R^d)}.
\end{equation}

\vskip 0.2 true cm

\noindent\textbf{Estimate of $J_{kk_1k_2}^2(s)$:}
Applying \eqref{bilinear:b} and \eqref{bilinear:e} again yields
\begin{equation*}
\begin{split}
J_{kk_1k_2}^2(s)
&\ls2^{k_1(2d+3)+k_2}\|P_{k_1}U(s)\|_{L^\infty(\R^d)}\|P_{k_2}\cN(s)\|_{L^2(\R^d)}\\
&\ls2^{k_2}\|U(s)\|_{W^{2d+4,\infty}(\R^d)}\sum_{(k_3,k_4)\in\cX_{k_2}}\sum_{\mu,\nu=\pm}
\|P_{k_2}B_{a_{\mu\nu}}(P_{k_3}U_\mu,P_{k_4}U_\nu)(s)\|_{L^2(\R^d)}\\
&\ls\|U(s)\|_{W^{2d+4,\infty}(\R^d)}\sum_{(k_3,k_4)\in\cX_{k_2}}
2^{k_2+k_4}\|P_{k_3}U(s)\|_{L^\infty(\R^d)}\|P_{k_4}U(s)\|_{L^2(\R^d)}\\
&\ls2^{-k_2}\|U(s)\|_{W^{2d+4,\infty}(\R^d)}\|U(s)\|_{W^{4,\infty}(\R^d)}\|U(s)\|_{H^4(\R^d)}.
\end{split}
\end{equation*}
Then
\begin{equation}\label{low:energy5}
\Big\|\sum_{(k_1,k_2)\in\cX_k}J_{kk_1k_2}^2(s)\Big\|_{\ell^2_k}
\ls\|U(s)\|^2_{W^{2d+4,\infty}(\R^d)}\|U(s)\|_{H^N(\R^d)}.
\end{equation}
In addition, the estimate on the second line in \eqref{low:energy1} is analogously.
Therefore, collecting \eqref{low:energy1}--\eqref{low:energy5} leads to
\begin{equation*}
\begin{split}
\|V(t)-V(1)\|_{L^2(\R^d)}&\ls\big\|\|P_k(V(t)-V(1))\|_{L^2}\big\|_{\ell^2_k}\\
&\ls\|U(1)\|^2_{H^N(\R^d)}+\|U(t)\|^2_{H^N(\R^d)}
+\int_1^t\|U(s)\|^2_{W^{2d+4,\infty}(\R^d)}\|U(s)\|_{H^N(\R^d)}ds.
\end{split}
\end{equation*}
On the other hand, one has
\begin{equation*}
\|U(t)\|_{L^2(\R^d)}\ls\|V(t)\|_{L^2(\R^d)}\ls\|V(1)\|_{L^2(\R^d)}
+\|V(t)-V(1)\|_{L^2(\R^d)}.
\end{equation*}
Consequently, \eqref{low:energy} is proved.
\end{proof}

\subsection{Proof of Theorem \ref{thm1}}
\begin{proof}[Proof of Theorem \ref{thm1}]
Suppose that for any $t\in[1,T_\ve)$,
\begin{equation*}
\|U(t)\|_{H^N(\R^d)}\le\ve_1.
\end{equation*}
By \eqref{initial:data}, the Strichartz estimate \eqref{Strichartz}, energy estimate Proposition \ref{prop:high:energy},
Lemmas \ref{lem:goodunknown} and \ref{lem:low:energy}, there is a constant $C_1\ge1$ such that
\begin{equation}\label{thm1:pf}
\begin{split}
&\quad\;\|U(t)\|_{H^N(\R^d)}\\
&\ls\ve+\ve_1^2+\ve_1\sum_{k_1,k_2}2^{(k_1+k_2)(2d+5+1/8)}
\|P_{k_1}U(s)\|_{L^2([1,t])L^\infty(\R^d)}\|P_{k_2}U(s)\|_{L^2([1,t])L^\infty(\R^d)}\\
&\ls\ve+\ve_1^2+\ve_1c_d^2(t)\sum_{k_1,k_2}2^{(k_1+k_2)(2d+5+1/8+d/2)}
\|P_{k_1}U\|_{L^2(\R^d)}\|P_{k_2}U\|_{L^2(\R^d)}\\
&\le C_1(\ve+\ve_1^2+\ve_1^3c_d^2(t)),
\end{split}
\end{equation}
where $N>2d+5+1/8+d/2$ is used.
Note that for $t\in[1,T_\ve)$, $c_d^2(t)=1$ when $d\ge3$, $c_2^2(t)\le\ln t\le\frac{\kappa}{\ve^2}$.
Choosing $\ve_0=\frac{1}{16C_1^2}$, $\kappa=\frac{1}{64C_1^3}$ and $\ve_1=4C_1$, then it follows from \eqref{thm1:pf} that
\begin{equation*}
\|U(t)\|_{H^N(\R^d)}\le3\ve_1/4.
\end{equation*}
This, together with the local existence of classical solution to \eqref{KG}, ensures that \eqref{KG} admits a unique solution $u\in C([0,T_\ve),H^{N+1}(\R^d))\cap C^1([0,T_\ve),H^N(\R^d))$.
\end{proof}
\begin{remark}\label{thm1pf:d=1}
For $d=1$, set $T_\ve=\kappa^2/\ve^4$. Note that for $t\in[1,T_\ve)$ one has $c_1^2(t)\le t^{1/2}\le\frac{\kappa}{\ve^2}$.
Then for $N\ge8$, Theorem \ref{thm1} holds for $d=1$ with $T_\ve=\kappa^2/\ve^4$.
\end{remark}

\section{Weighted $L^2$ norm estimate and proof of Theorem \ref{thm2}}

In this section, we restrict $d=2$ in problem \eqref{KG}.
Suppose that for $N\ge12$,  $\alpha\in(0,1/5)$ and $t\ge1$,
\begin{equation}\label{BA:d=2}
\|U(t)\|_{H^{N}(\R^2)}+\|\w{x}^\alpha V(t)\|_{L^2(\R^2)}\le\ve_2,
\end{equation}
where $U,V$ are defined by \eqref{U:def} and \eqref{profile:def}.
Define the dyadic decomposition in the Euclidean physical space $\R^2$
\begin{equation}
(Q_jf)(x):=\psi_j(x)f(x),\quad j\in\Z,j\ge-1.
\end{equation}

\begin{lemma}
Suppose that $V$ is defined by \eqref{profile:def}, for any $\alpha\in(0,1/5)$, we have
\begin{equation}\label{Znorm:local}
2^{j\alpha}\|Q_jP_kV\|_{L^2(\R^2)}\ls\|\w{x}^\alpha V\|_{L^2(\R^2)}
\ls\big\|2^{j\alpha}\|Q_jP_kV\|_{L^2(\R^2)}\big\|_{\ell^1_k\ell^2_j}.
\end{equation}
\end{lemma}
\begin{proof}
It is obvious that $P_k$ is a bounded operator with $\|P_kV\|_{L^2(\R^2)}\ls\|V\|_{L^2(\R^2)}$.
On the other hand, $\w{x}^{\alpha}$ belongs to $A_2$ class (see \cite{Stein}) and one can achieve $\|\w{x}^\alpha P_kV\|_{L^2(\R^2)}\ls\|\w{x}^{\alpha}V\|_{L^2(\R^2)}$, which yields
\begin{equation*}
2^{j\alpha}\|Q_jP_kV\|_{L^2(\R^2)}\ls\|\w{x}^\alpha P_kV\|_{L^2(\R^2)}
\ls\|\w{x}^\alpha V\|_{L^2(\R^2)}.
\end{equation*}
Thus, we have proved the first inequality in \eqref{Znorm:local}.
The second inequality in \eqref{Znorm:local} can be obtained by the Minkowski inequality
\begin{equation*}
\|\w{x}^\alpha V\|_{L^2(\R^2)}\ls\|\sum_{k\ge-1}\w{x}^\alpha P_kV\|_{L^2(\R^2)}
\ls\sum_{k\ge-1}\|\w{x}^\alpha P_kV\|_{L^2(\R^2)}
\ls\sum_{k\ge-1}\big\|2^{j\alpha}\|Q_jP_kV\|_{L^2(\R^2)}\big\|_{\ell^2_j}.
\end{equation*}
\end{proof}
It follows from the first term in \eqref{BA:d=2} that
\begin{equation*}
\|Q_jP_kV(t)\|_{L^2(\R^2)}\ls\|P_kV(t)\|_{L^2(\R^2)}\ls2^{-Nk}\ve_2.
\end{equation*}
Interpolating this inequality with \eqref{Znorm:local} yields that for any $n\in[0,N]$,
\begin{equation}\label{BA:d=2'}
\|Q_jP_kV(t)\|_{L^2(\R^2)}\ls2^{-j\alpha(1-n/N)-nk}\ve_2.
\end{equation}

\subsection{Localized dispersive estimate and Strichartz estimate}

\begin{lemma}[Localized dispersive estimate]
Suppose that $U,V$ are defined by \eqref{U:def}, \eqref{profile:def} and the bootstrap assumption \eqref{BA:d=2} holds.
For any $n_1,n_2\in[0,N]$ and $t\ge1$, one has
\begin{equation}\label{local:disp}
\|e^{it\mu\Lambda}P_{[k-1,k+1]}Q_jP_kV_\mu\|_{L^\infty(\R^2)}
\ls2^{k(1-n_1+\alpha(1-n_2/N)+j\alpha(n_1-n_2)/N}t^{-\alpha(1-n_2/N)}\ve_2.
\end{equation}
\end{lemma}
\begin{proof}
By the Bernstein inequality, we have
\begin{equation}\label{local:disp1}
\|e^{it\mu\Lambda}P_{[k-1,k+1]}Q_jP_kV_\mu\|_{L^\infty(\R^2)}
\ls2^k\|e^{it\mu\Lambda}Q_jP_kV_\mu\|_{L^2(\R^2)}\ls2^k\|Q_jP_kV\|_{L^2(\R^2)}.
\end{equation}
On the other hand, it can be deduced from \eqref{disp:estimate} that
\begin{equation}\label{local:disp2}
\|e^{it\mu\Lambda}P_{[k-1,k+1]}Q_jP_kV_\mu\|_{L^\infty(\R^2)}
\ls2^{2k}t^{-1}\|Q_jP_kV\|_{L^1(\R^2)}
\ls2^{2k+j}t^{-1}\|Q_jP_kV\|_{L^2(\R^2)}.
\end{equation}
Interpolation between \eqref{local:disp1} and \eqref{local:disp2} leads to
\begin{equation*}
\|e^{it\mu\Lambda}P_{[k-1,k+1]}Q_jP_kV_\mu\|_{L^\infty(\R^2)}
\ls2^{k+\alpha(1-n_2/N)(k+j)}t^{-\alpha(1-n_2/N)}\|Q_jP_kV\|_{L^2(\R^2)}.
\end{equation*}
This, together with \eqref{BA:d=2'}, yields \eqref{local:disp}.
\end{proof}

\begin{lemma}[Localized Strichartz estimate]
Suppose that $U,V$ are defined by \eqref{U:def}, \eqref{profile:def} and the bootstrap assumption \eqref{BA:d=2} holds.
For any $0\le\beta_1<\beta_2\le1$, $n\in[0,N]$ and $t\ge1$, one has
\begin{equation}\label{local:Stric}
\|s^{\beta_1/2}e^{is\mu\Lambda}P_{[k-1,k+1]}Q_jP_kV_\mu\|_{L^2([1,t])L^\infty(\R^2)}
\ls2^{k(1+\beta_2-n)+j\beta_2-j\alpha(1-n/N)}\ve_2.
\end{equation}
\end{lemma}
\begin{proof}
\eqref{local:disp1} ensures that for any $p\in(2,\infty)$,
\begin{equation*}
\|s^{1/2}P_{[k-1,k+1]}e^{is\mu\Lambda}Q_jP_kV_\mu\|_{L^p([1,t])L^\infty(\R^2)}
\ls\frac{2^{2k+j}}{(p-2)^{1/p}}\|Q_jP_kV_\mu\|_{L^2(\R^2)}.
\end{equation*}
Interpolating this inequality with \eqref{Strichartz'} yields
\begin{equation*}
\|s^{\beta_2/2}P_{[k-1,k+1]}e^{is\mu\Lambda}Q_jP_kV_\mu\|_{L^p([1,t])L^\infty(\R^2)}
\ls\frac{2^{k+\beta_2(k+j)}}{(p-2)^{1/p}}\|Q_jP_kV_\mu\|_{L^2(\R^2)}.
\end{equation*}
Choosing $p=\frac{2}{1-(\beta_2-\beta_1)/2}\in(2,\frac{2}{1-(\beta_2-\beta_1)})$.
Then we can conclude from the H\"{o}lder inequality that
\begin{equation*}
\begin{split}
&\quad\;\|s^{\beta_1/2}P_{[k-1,k+1]}e^{is\mu\Lambda}Q_jP_kV_\mu\|_{L^2([1,t])L^\infty(\R^2)}\\
&\ls\|s^{(\beta_1-\beta_2)/2}\|_{L^\frac{2p}{p-2}([1,t])}
\|s^{\beta_2/2}P_{[k-1,k+1]}e^{is\mu\Lambda}Q_jP_kV_\mu\|_{L^p([1,t])L^\infty(\R^2)}\\
&\ls2^{k+\beta_2(k+j)}\|Q_jP_kV_\mu\|_{L^2(\R^2)}.
\end{split}
\end{equation*}
This, together with \eqref{BA:d=2'}, leads to \eqref{local:Stric}.
\end{proof}

\subsection{Weighted $L^2$ norm estimate}

\begin{lemma}\label{lem:Znorm1}
Suppose that $U,V$ are defined by \eqref{U:def}, \eqref{profile:def} and the bootstrap assumption \eqref{BA:d=2} holds, then we have
\begin{equation}\label{Znorm:bdry}
\|\w{x}^\alpha e^{-it\Lambda}B_{\Phi^{-1}_{\mu\nu}a_{\mu\nu}}(U_\mu,U_\nu)(t)\|_{L^2(\R^2)}
\ls\ve_2^2.
\end{equation}
\end{lemma}
\begin{proof}
Due to \eqref{Znorm:local}, it only suffices to show
\begin{equation*}
\big\|2^{j\alpha}\|Q_jP_ke^{-it\Lambda}B_{\Phi^{-1}_{\mu\nu}a_{\mu\nu}}
(U_\mu,U_\nu)(t)\|_{L^2(\R^2)}\big\|_{\ell^1_k\ell^2_j}\ls\ve_2^2.
\end{equation*}
By virtue of \eqref{proj:proj}, we can find that
\begin{equation}\label{Znorm:bdry1}
\begin{split}
&Q_jP_ke^{-it\Lambda}B_{\Phi^{-1}_{\mu\nu}a_{\mu\nu}}(U_\mu,U_\nu)
=\sum_{j_1,j_2\ge-1}\sum_{(k_1,k_2)\in\cX_k}I_{kk_1k_2}^{jj_1j_2},\\
&I_{kk_1k_2}^{jj_1j_2}:=Q_jP_ke^{-it\Lambda}B_{\Phi^{-1}_{\mu\nu}a_{\mu\nu}}
(e^{it\mu\Lambda}P_{[[k_1]]}Q_{j_1}P_{k_1}V_\mu,
e^{it\nu\Lambda}P_{[[k_2]]}Q_{j_2}P_{k_2}V_\nu),
\end{split}
\end{equation}
where $[[k]]:=[k-1,k+1]$.
We only require to deal with the case of $k_1\le k_2$ in \eqref{Znorm:bdry1}
since the case of $k_1\ge k_2$ can be treated analogously.

\vskip 0.2 true cm
\noindent\textbf{Case 1. $j\ge\log_2t+10$}
\vskip 0.1 true cm

In this case, $I_{kk_1k_2}^{jj_1j_2}$ can be recast as
\begin{equation*}
\begin{split}
I_{kk_1k_2}^{jj_1j_2}(t,x)
&=(2\pi)^{-4}\psi_j(x)\iint_{(\R^2)^2}K(x-y,x-z)Q_{j_1}P_{k_1}V_\mu(t,y)Q_{j_2}P_{k_2}V_\nu(t,z)dydz,\\
\end{split}
\end{equation*}
where
\begin{equation*}
\begin{split}
K(x-y,x-z)&:=\iint_{(\R^2)^2}e^{i\tilde\Phi}(\Phi^{-1}_{\mu\nu}a_{\mu\nu})(\xi,\eta)
\psi_k(\xi+\eta)\psi_{[[k_1]]}(\xi)\psi_{[[k_2]]}(\eta)d\xi d\eta,\\
\tilde\Phi&:=\xi\cdot(x-y)+\eta\cdot(x-z)+t\Phi_{\mu\nu}(\xi,\eta).
\end{split}
\end{equation*}
By Lemma \ref{lem:phase}, for $\xi+\eta\in\supp\psi_k$, $\xi\in\supp\psi_{[[k_1]]}$, $\eta\in\supp\psi_{[[k_2]]}$ and $(k_1,k_2)\in\cX_k$ we have
\begin{equation*}
\begin{split}
&|\p_\xi\Phi_{\mu\nu}(\xi,\eta)|+|\p_\eta\Phi_{\mu\nu}(\xi,\eta)|\le4,\\
&|\p_{\xi,\eta}^l\Phi_{\mu\nu}(\xi,\eta)|\ls1,\quad
|\p_{\xi,\eta}^l\Phi^{-1}_{\mu\nu}(\xi,\eta)|\ls|\Phi^{-1}_{\mu\nu}|\ls2^{k_1},\quad l\ge1.
\end{split}
\end{equation*}
If $\max\{|j-j_1|,|j-j_2|\}\ge5$, for $x\in\supp\psi_j,y\in\supp\psi_{j_1},z\in\supp\psi_{j_2}$, one then has
\begin{equation*}
|x-y|+|x-z|\ge8t,\qquad2^{\max\{j,j_1,j_2\}}\ls|x-y|+|x-z|.
\end{equation*}
This ensures
\begin{equation*}
\max\{t,2^{\max\{j,j_1,j_2\}}\}\ls|x-y|+|x-z|
\ls|\p_\xi\tilde\Phi|+|\p_\eta\tilde\Phi|.
\end{equation*}
Let
\begin{equation*}
\begin{split}
L&:=-i(|\p_\xi\tilde\Phi|^2+|\p_\eta\tilde\Phi|^2)^{-1}
\sum_{l=1}^2(\p_{\xi_l}\tilde\Phi\p_{\xi_l}+\p_{\eta_l}\tilde\Phi\p_{\eta_l}),\\
L^*&:=i\sum_{l=1}^2\{\p_{\xi_l}(\frac{\p_{\xi_l}\tilde\Phi~\cdot}
{|\p_\xi\tilde\Phi|^2+|\p_\eta\tilde\Phi|^2}
+\p_{\eta_l}(\frac{\p_{\eta_l}\tilde\Phi~\cdot}
{|\p_\xi\tilde\Phi|^2+|\p_\eta\tilde\Phi|^2})\},
\end{split}
\end{equation*}
then one has $Le^{i\tilde\Phi}=e^{i\tilde\Phi}$.
It follows from the method of stationary phase that
\begin{equation*}
\begin{split}
&\quad\;|K(x-y,x-z)|\\
&=\Big|\iint_{(\R^2)^2}L^6(e^{i\tilde\Phi})(\Phi^{-1}_{\mu\nu}a_{\mu\nu})(\xi,\eta)
\psi_k(\xi+\eta)\psi_{[[k_1]]}(\xi)\psi_{[[k_2]]}(\eta)d\xi d\eta\Big|\\
&=\Big|\iint_{(\R^2)^2}e^{i\tilde\Phi}(L^*)^6\Big\{(\Phi^{-1}_{\mu\nu}a_{\mu\nu})(\xi,\eta)
\psi_k(\xi+\eta)\psi_{[[k_1]]}(\xi)\psi_{[[k_2]]}(\eta)\Big\}d\xi d\eta\Big|\\
&\ls2^{k_1+k_2-\max\{j,j_1,j_2\}}(1+|x-y|+|x-z|)^{-5}
(\sum_{l\le6}\|\p^l\psi_{[[k_1]]}\|_{L^1})(\sum_{l\le6}\|\p^l\psi_{[[k_2]]}\|_{L^1})\\
&\ls2^{3k_1+3k_2-\max\{j,j_1,j_2\}}(1+|x-y|+|x-z|)^{-5},
\end{split}
\end{equation*}
which yields
\begin{equation*}
\|K(y,z)\|_{L^1((\R^2)^2)}\ls2^{3k_1+3k_2-\max\{j,j_1,j_2\}}.
\end{equation*}
This, together with the H\"{o}lder inequality, the Bernstein inequality and \eqref{BA:d=2'} with $n=N$, leads to
\begin{equation*}
\begin{split}
\|I_{kk_1k_2}^{jj_1j_2}(t)\|_{L^2(\R^2)}&\ls\|K(y,z)\|_{L^1((\R^2)^2)}
\|P_{k_1}V_\mu\|_{L^\infty(\R^2)}\|P_{k_2}V_\nu\|_{L^2(\R^2)}\\
&\ls2^{4k_1+3k_2-\max\{j,j_1,j_2\}}\|P_{k_1}V_\mu\|_{L^2(\R^2)}\|P_{k_2}V_\nu\|_{L^2(\R^2)}\\
&\ls2^{k_1(4-N)+k_2(3-N)-\max\{j,j_1,j_2\}}\ve_2^2.
\end{split}
\end{equation*}
Therefore, we can obtain that for $\alpha\in(0,1/5)$ and $N\ge12$,
\begin{equation}\label{Znorm:bdry2}
\begin{split}
&\quad\;\Big\|2^{j\alpha}\sum_{\substack{j_1,j_2\ge-1,\\\max\{|j-j_1|,|j-j_2|\}\ge5}}
\sum_{(k_1,k_2)\in\cX_k}\|I_{kk_1k_2}^{jj_1j_2}(t)\|_{L^2(\R^2)}
\Big\|_{\ell^1_k\ell^2_j(j\ge\log_2t+10)}\\
&\ls\Big\|\ve_2^2\sum_{j_1,j_2\ge-1}\sum_{(k_1,k_2)\in\cX_k}
2^{k_2(3-N)+\max\{j,j_1,j_2\}(\alpha-1)}\Big\|_{\ell^1_k\ell^2_j}\ls\ve_2^2.
\end{split}
\end{equation}
It remains to deal with the case of $\max\{|j-j_1|,|j-j_2|\}\le4$ in \eqref{Znorm:bdry1}.
By \eqref{bilinear:e}, \eqref{BA:d=2'} with $n=10$ and \eqref{local:disp} with $n_1=0$, $n_2=N$, we can get that
\begin{equation}\label{Znorm:bdry3}
\begin{split}
2^{j\alpha}\|I_{kk_1k_2}^{jj_1j_2}(t)\|_{L_x^2}
&\ls 2^{j\alpha}\|B_{\Phi^{-1}_{\mu\nu}a_{\mu\nu}}
(e^{it\mu\Lambda}P_{[[k_1]]}Q_{j_1}P_{k_1}V_\mu,
e^{it\nu\Lambda}P_{[[k_2]]}Q_{j_2}P_{k_2}V_\nu)\|_{L_x^2}\\
&\ls2^{7k_1+k_2+j\alpha}\|e^{it\mu\Lambda}P_{[[k_1]]}Q_{j_1}P_{k_1}V_\mu\|_{L_x^\infty}
\|Q_{j_2}P_{k_2}V_\nu\|_{L_x^2}\\
&\ls2^{8k_1-9k_2+j\alpha-j_1\alpha-j_2\alpha(1-10/N)}\ve_2^2
\ls2^{-k_2-j\alpha(1-10/N)}\ve_2^2.
\end{split}
\end{equation}
This gives
\begin{equation}\label{Znorm:bdry4}
\begin{split}
&\quad\;\Big\|2^{j\alpha}\sum_{\substack{j_1,j_2\ge-1,\\\max\{|j-j_1|,|j-j_2|\}\le4}}
\sum_{(k_1,k_2)\in\cX_k}\|I_{kk_1k_2}^{jj_1j_2}(t)\|_{L_x^2}
\Big\|_{\ell^1_k\ell^2_j(j\ge\log_2t+10)}\\
&\ls\Big\|\ve_2^2\sum_{j_1,j_2\ge-1}\sum_{(k_1,k_2)\in\cX_k}
2^{-k_2-\alpha(1-10/N)\max\{j,j_1,j_2\}}\Big\|_{\ell^1_k\ell^2_j}\ls\ve_2^2.
\end{split}
\end{equation}

\vskip 0.2 true cm
\noindent\textbf{Case 2. $j\le\log_2t+10$ and $j_2\ge\log_2t$}
\vskip 0.1 true cm

Similarly to \eqref{Znorm:bdry3} with $n_2=1$ in \eqref{local:disp}, we have
\begin{equation*}
\begin{split}
2^{j\alpha}\|I_{kk_1k_2}^{jj_1j_2}(t)\|_{L_x^2}
&\ls2^{7k_1+k_2+j\alpha}\|e^{it\mu\Lambda}P_{[[k_1]]}Q_{j_1}P_{k_1}V_\mu\|_{L_x^\infty}
\|Q_{j_2}P_{k_2}V_\nu\|_{L_x^2}\\
&\ls\ve_2^22^{k_1(8+\alpha(1-1/N))-9k_2+j\alpha-j_1\alpha/N-j_2\alpha(1-10/N)}
t^{-\alpha(1-1/N)}\\
&\ls\ve_2^22^{-k_2(1-\alpha)-j_1\alpha/N-j_2\alpha(1-11/N)}.
\end{split}
\end{equation*}
This leads to
\begin{equation}\label{Znorm:bdry5}
\Big\|2^{j\alpha}\sum_{\substack{j_1,j_2\ge-1,\\j_2\ge\ln t}}
\sum_{(k_1,k_2)\in\cX_k}\|I_{kk_1k_2}^{jj_1j_2}(t)\|_{L_x^2}
\Big\|_{\ell^1_k\ell^2_j(j\le\log_2t+10)}\ls\ve_2^2.
\end{equation}

\vskip 0.2 true cm
\noindent\textbf{Case 3. $j\le\log_2t+10$, $j_2\le\log_2t$ and $j_1\le j_2$}
\vskip 0.1 true cm

Applying \eqref{disp:estimate} instead of \eqref{local:disp} in \eqref{Znorm:bdry3} gives
\begin{equation}\label{Znorm:bdry6}
\begin{split}
2^{j\alpha}\|I_{kk_1k_2}^{jj_1j_2}(t)\|_{L_x^2}
&\ls2^{7k_1+k_2+j\alpha}\|e^{it\mu\Lambda}P_{[[k_1]]}Q_{j_1}P_{k_1}V_\mu\|_{L_x^\infty}
\|Q_{j_2}P_{k_2}V_\nu\|_{L_x^2}\\
&\ls2^{9k_1+k_2+j\alpha}t^{-1}\|Q_{j_1}P_{k_1}V_\mu\|_{L_x^1}\|Q_{j_2}P_{k_2}V_\nu\|_{L_x^2}\\
&\ls2^{9k_1+k_2+j_1}t^{-1+\alpha}\|Q_{j_1}P_{k_1}V_\mu\|_{L_x^2}\|Q_{j_2}P_{k_2}V_\nu\|_{L_x^2}\\
&\ls\ve_2^22^{-k_2+j_1-\alpha(j_1+j_2)(1-\frac{11}{2N})}t^{-1+\alpha}\\
&\ls\ve_2^22^{-k_2+j_1(1-2\alpha+11\alpha/N)}t^{-1+\alpha}
\ls2^{-k_2-\alpha\max\{j,j_1,j_2\}(1-11/N)},
\end{split}
\end{equation}
where we have used \eqref{BA:d=2'} for $\|Q_{j_1}P_{k_1}V_\mu\|_{L_x^2}$ and $\|Q_{j_2}P_{k_2}V_\nu\|_{L_x^2}$ with $n=11/2$.
Therefore,
\begin{equation}\label{Znorm:bdry7}
\Big\|2^{j\alpha}\sum_{\substack{j_1,j_2\ge-1,\\j_1\le j_2\le\ln t}}
\sum_{(k_1,k_2)\in\cX_k}\|I_{kk_1k_2}^{jj_1j_2}(t)\|_{L_x^2}
\Big\|_{\ell^1_k\ell^2_j(j\le\log_2t+10)}\ls\ve_2^2.
\end{equation}

\vskip 0.2 true cm
\noindent\textbf{Case 4. $j\le\log_2t+10$, $j_2\le\log_2t$ and $j_1\ge j_2$}
\vskip 0.1 true cm

Changing the corresponding $L^\infty$ and $L^2$ norms in \eqref{Znorm:bdry6} to obtain
\begin{equation*}
\begin{split}
2^{j\alpha}\|I_{kk_1k_2}^{jj_1j_2}(t)\|_{L_x^2}
&\ls\ve_2^22^{-k_2+j_2-\alpha(j_1+j_2)(1-\frac{11}{2N})}t^{-1+\alpha}\\
&\ls\ve_2^22^{-k_2-\frac{j_1\alpha}{2N}+j_2(1-2\alpha+\alpha\frac{23}{2N})}t^{-1+\alpha}\\
&\ls\ve_2^22^{-k_2-\frac{j_1\alpha}{2N}-\alpha\max\{j,j_2\}(1-\frac{23}{2N})}.
\end{split}
\end{equation*}
This, together with $N\ge12$, ensures
\begin{equation}\label{Znorm:bdry8}
\Big\|2^{j\alpha}\sum_{\substack{j_1,j_2\ge-1,\\j_2\le\ln t,j_1\ge j_2}}
\sum_{(k_1,k_2)\in\cX_k}\|I_{kk_1k_2}^{jj_1j_2}(t)\|_{L_x^2}
\Big\|_{\ell^1_k\ell^2_j(j\le\log_2t+10)}\ls\ve_2^2.
\end{equation}
Finally, substituting \eqref{Znorm:bdry2}--\eqref{Znorm:bdry8} into \eqref{Znorm:bdry1} completes the proof of \eqref{Znorm:bdry}.
\end{proof}

\subsection{Weighted $L^2$ norm estimate of the nonlinearity}

At first, we will give another formulation of the second line in \eqref{profile:eqn2}.
Due to \eqref{bilinear:def} and \eqref{U:eqn}, we have
\begin{equation*}
\p_t\hat V_\sigma(t,\xi)=\sum_{\mu,\nu=\pm}\int_{\R^2}e^{-it\sigma\Lambda(\xi)}
a_{\mu\nu}(\xi-\eta,\eta)\hat U_\mu(t,\xi-\eta)\hat U_\nu(t,\eta)d\eta
+e^{-it\sigma\Lambda(\xi)}\hat\cC_1(t,\xi).
\end{equation*}
Then the second line of \eqref{profile:eqn1} can be reduced to
\begin{equation*}
\begin{split}
&\quad\; i\sum_{\mu,\nu=\pm}\int_1^t\iint_{\R^4}e^{-is\Lambda(\xi)}
(\Phi^{-1}_{\mu\nu}a_{\mu\nu})(\xi-\eta,\eta)\Big(
a_{\sigma\iota}(\xi-\eta-\zeta,\zeta)
\hat U_\sigma(\xi-\eta-\zeta)\hat U_\iota(\zeta)\hat U_\nu(\eta)\\
&\hspace{2cm}+\hat U_\mu(\xi-\eta)a_{\sigma\iota}(\eta-\zeta,\zeta)
\hat U_\sigma(\eta-\zeta)\hat U_\iota(\zeta)\Big)d\eta d\zeta ds
+\int_1^te^{-is\Lambda(\xi)}\hat\cC_2(s,\xi)ds\\
&=i\sum_{\mu,\sigma,\iota=\pm}\int_1^t\iint_{\R^4}e^{-is\Lambda(\xi)}
b_{\mu\sigma\iota}(\xi-\eta,\eta-\zeta,\zeta)\hat U_\mu(\xi-\eta)
\hat U_\sigma(\eta-\zeta)\hat U_\iota(\zeta)d\eta d\zeta ds
+\int_1^te^{-is\Lambda(\xi)}\hat\cC_2(s,\xi)ds,
\end{split}
\end{equation*}
where $\cC_2$ is at least quartic in $U$ and
\begin{equation}\label{symbol:b}
b=b_{\mu\sigma\iota}(\xi-\eta,\eta-\zeta,\zeta)=\sum_{\nu=\pm}
a_{\sigma\iota}(\eta-\zeta,\zeta)((\Phi^{-1}_{\mu\nu}a_{\mu\nu})(\xi-\eta,\eta)
+(\Phi_{\nu\mu}^{-1}a_{\nu\mu})(\xi,\xi-\eta)).
\end{equation}
Therefore, it concludes that
\begin{equation}\label{profile:eqn3}
\begin{split}
V(t,x)&=V(1,x)-i\sum_{\mu,\nu=\pm}e^{-is\Lambda}
B_{\Phi^{-1}_{\mu\nu}a_{\mu\nu}}(U_\mu,U_\nu)(s,x)\Big|_{s=1}^t\\
&\quad+i\sum_{\mu,\sigma,\iota=\pm}\int_1^te^{-is\Lambda}
\sT_b(U_\mu,U_\sigma,U_\iota)(s,x)ds+\int_1^te^{-is\Lambda}\cC_3ds,
\end{split}
\end{equation}
where the trilinear pseudoproduct operator $\sT_b(U_\mu,U_\sigma,U_\iota)$ is defined by \eqref{trilinear:def} and $\cC_3=\cC_1+\cC_2$.
Note that the estimate of $\cC_3$ is easier to be obtained than that for $\sT_b(U_\mu,U_\sigma,U_\iota)$, then we omit it here.
\begin{lemma}
Suppose that $U,V$ are defined by \eqref{U:def}, \eqref{profile:def} and the bootstrap assumption \eqref{BA:d=2} holds, then
\begin{equation}\label{Znorm}
\Big\|\w{x}^\alpha\int_1^te^{-is\Lambda}\sT_b(U_\mu,U_\sigma,U_\iota)(s)ds\Big\|_{L^2(\R^2)}
\ls\ve_2^3.
\end{equation}
\end{lemma}
\begin{proof}
Denote
\begin{equation*}
\begin{split}
\cY_k&=\cY_k^1\cup\cY_k^2,\\
\cY_k^1&=\{(k_1,k_2,k_3)\in\Z^3: |\max\{k_1,k_2,k_3\}-k|\le4,k_1,k_2,k_3\ge-1\},\\
\cY_k^2&=\{(k_1,k_2,k_3)\in\Z^3: \max\{k_1,k_2,k_3\}\ge k+4,
\max\{k_1,k_2,k_3\}-{\rm med}\{k_1,k_2,k_3\}\le4,k_1,k_2,k_3\ge-1\}.
\end{split}
\end{equation*}
As in \cite[page 799]{IP13}, if $P_k(P_{k_1}fP_{k_2}gP_{k_3}h)\neq0$, then $(k_1,k_2,k_3)\in\cY_k$.
Similarly to \eqref{Znorm:bdry1}, one has
\begin{equation}\label{Znorm1}
\begin{split}
&Q_jP_ke^{-is\Lambda}\sT_b(U_\mu,U_\sigma,U_\iota)=\sum_{j_1,j_2,j_3\ge-1}
\sum_{(k_1,k_2,k_3)\in\cY_k}I_{kk_1k_2k_3}^{jj_1j_2j_3},\\
&I_{kk_1k_2k_3}^{jj_1j_2j_3}:=Q_jP_ke^{-is\Lambda}\sT_b(e^{is\mu\Lambda}P_{[[k_1]]}Q_{j_1}P_{k_1}V_\mu,
e^{is\sigma\Lambda}P_{[[k_2]]}Q_{j_2}P_{k_2}V_\sigma,e^{is\iota\Lambda}P_{[[k_3]]}Q_{j_3}P_{k_3}V_\iota).
\end{split}
\end{equation}
Without loss of generality, we only deal with the case of $k_1\le k_2\le k_3$ in \eqref{Znorm1}.

\vskip 0.2 true cm
\noindent\textbf{Case 1. In the set $\cI_1:=\{j\ge\log_2s+20\}$}
\vskip 0.1 true cm

This case is similar to the Case 1 in Lemma \ref{lem:Znorm1}.
$I_{kk_1k_2k_3}^{jj_1j_2j_3}$ can be recast as
\begin{equation*}
\begin{split}
&I_{kk_1k_2k_3}^{jj_1j_2j_3}(s,x)=(2\pi)^{-6}\psi_j(x)\iiint_{(\R^2)^3}
K(x-x_1,x-x_2,x-x_3)Q_{j_1}P_{k_1}V_\mu(s,x_1)\\
&\hspace{5cm}\times Q_{j_2}P_{k_2}V_\sigma(s,x_2)
Q_{j_3}P_{k_3}V_\iota(s,x_3)dx_1dx_2dx_3,\\
\end{split}
\end{equation*}
where
\begin{equation*}
\begin{split}
&K(x-x_1,x-x_2,x-x_3):=\iiint_{(\R^2)^3}e^{i\tilde\Psi}b_{\mu\sigma\iota}(\xi,\eta,\zeta)
\psi_k(\xi+\eta+\zeta)\psi_{[[k_1]]}(\xi)\psi_{[[k_2]]}(\eta)
\psi_{[[k_3]]}(\zeta)d\xi d\eta d\zeta,\\
&\tilde\Psi:=s\Psi_{\mu\sigma\iota}(\xi,\eta,\zeta)
+\xi\cdot(x-x_1)+\eta\cdot(x-x_2)+\zeta\cdot(x-x_3),\\
&\Psi_{\mu\sigma\iota}(\xi,\eta,\zeta):=-\Lambda(\xi+\eta+\zeta)+\mu\Lambda(\xi)
+\sigma\Lambda(\eta)+\iota\Lambda(\zeta).
\end{split}
\end{equation*}
Denote
\begin{equation*}
\cL:=-i(|\p_\xi\tilde\Psi|^2+|\p_\eta\tilde\Psi|^2+|\p_\zeta\tilde\Psi|^2)^{-1}
\sum_{l=1}^2(\p_{\xi_l}\tilde\Psi\p_{\xi_l}+\p_{\eta_l}\tilde\Psi\p_{\eta_l}
+\p_{\zeta_l}\tilde\Psi\p_{\zeta_l}).
\end{equation*}
It follows from the method of stationary phase that
\begin{equation*}
\begin{split}
&\quad\;|K(x-x_1,x-x_2,x-x_3)|\\
&=\Big|\iiint_{(\R^2)^3}\cL^{10}(e^{i\tilde\Psi})b_{\mu\sigma\iota}(\xi,\eta,\zeta)
\psi_k(\xi+\eta+\zeta)\psi_{[[k_1]]}(\xi)\psi_{[[k_2]]}(\eta)
\psi_{[[k_3]]}(\zeta)d\xi d\eta d\zeta\Big|\\
&\ls2^{2k_1+2k_2+5k_3-\max\{j,j_1,j_2,j_3\}}(1+|x-x_1|+|x-x_2|+|x-x_3|)^{-7}s^{-2},
\end{split}
\end{equation*}
which yields
\begin{equation*}
\|K(x_1,x_2,x_3)\|_{L^1((\R^2)^3)}\ls2^{2k_1+2k_2+5k_3-\max\{j,j_1,j_2,j_3\}}s^{-2}.
\end{equation*}
It can be deduced from the H\"{o}lder inequality and \eqref{BA:d=2'} with $n=N$ that
\begin{equation*}
\begin{split}
\|I_{kk_1k_2k_3}^{jj_1j_2j_3}(s)\|_{L^2(\R^2)}&\ls\|K(x_1,x_2,x_3)\|_{L^1((\R^2)^3)}
\|P_{k_1}V\|_{L^\infty(\R^2)}\|P_{k_2}V\|_{L^\infty(\R^2)}\|P_{k_3}V\|_{L^2(\R^2)}\\
&\ls2^{3k_1+3k_2+5k_3-\max\{j,j_1,j_2,j_3\}}\|P_{k_1}V\|_{L^2(\R^2)}
\|P_{k_2}V\|_{L^2(\R^2)}\|P_{k_3}V\|_{L^2(\R^2)}\\
&\ls2^{(k_1+k_2)(3-N)+k_3(5-N)-\max\{j,j_1,j_2,j_3\}}s^{-2}\ve_2^3.
\end{split}
\end{equation*}
Therefore, for $\max\{|j-j_l|,l=1,2,3\}\ge5$, one has
\begin{equation}\label{Znorm2}
\begin{split}
&\quad\;\Big\|2^{j\alpha}\sum_{\substack{j_1,j_2,j_3\ge-1,\\
\max\{|j-j_l|,l=1,2,3\}\ge5}}
\sum_{(k_1,k_2,k_3)\in\cY_k}\|I_{kk_1k_2k_3}^{jj_1j_2j_3}(s)\|_{L^2}
\Big\|_{L^1([1,t])\ell^1_k\ell^2_j(j\ge\log_2s+20)}\\
&\ls\Big\|\sum_{j_1,j_2,j_3\ge-1}\sum_{(k_1,k_2,k_3)\in\cY_k}
2^{k_3(5-N)+\max\{j,j_1,j_2,j_3\}(\alpha-1)}s^{-2}\ve_2^3
\Big\|_{L^1([1,t])\ell^1_k\ell^2_j}\ls\ve_2^3.
\end{split}
\end{equation}
Next, we focus on the case of $\max\{|j-j_l|,l=1,2,3\}\le4$.
By \eqref{trilinear}, \eqref{BA:d=2'} with $n=12$ and \eqref{local:Stric} with $\beta_1=0,\beta_2=\alpha/N,n=0$, we arrive at
\begin{equation}\label{Znorm3}
\begin{split}
2^{j\alpha}\|\id_{\cI_1}(s)I_{kk_1k_2k_3}^{jj_1j_2j_3}(s)\|_{L^1([1,t])L_x^2}
&\ls2^{2(k_1+k_2)+5k_3+j\alpha}\|e^{is\mu\Lambda}P_{[[k_1]]}Q_{j_1}P_{k_1}V_\mu\|_{L^2([1,t])L_x^\infty}\\
&\qquad\times\|e^{is\sigma\Lambda}P_{[[k_2]]}Q_{j_2}P_{k_2}V_\sigma\|_{L^2([1,t])L_x^\infty}
\|Q_{j_3}P_{k_3}V\|_{L_t^\infty L_x^2}\\
&\ls2^{k_3(2\beta_2-1)+j\alpha-\alpha(j_1+j_2)(1-1/N)-j_3\alpha(1-12/N)}\ve_2^3\\
&\ls2^{-k_3(1-2\alpha/N)-2j\alpha(1-7/N)}\ve_2^3,
\end{split}
\end{equation}
where $\id_{\cI_1}(s)$ is defined by \eqref{charact:def}.
This, together with the Minkowski inequality, gives
\begin{equation}\label{Znorm4}
\begin{split}
&\quad\;\Big\|2^{j\alpha}\sum_{\substack{j_1,j_2,j_3\ge-1,\\
\max\{|j-j_l|,l=1,2,3\}\le4}}
\sum_{(k_1,k_2,k_3)\in\cY_k}\|I_{kk_1k_2k_3}^{jj_1j_2j_3}(s)\|_{L_x^2}
\Big\|_{L^1([1,t])\ell^1_k\ell^2_j(j\ge\log_2s+20)}\\
&\ls\Big\|2^{j\alpha}\sum_{\substack{j_1,j_2,j_3\ge-1,\\\max\{|j-j_l|,l=1,2,3\}\le4}}
\sum_{(k_1,k_2,k_3)\in\cY_k}\|\id_{\cI_1}(s)
I_{kk_1k_2k_3}^{jj_1j_2j_3}(s)\|_{L^1([1,t])L_x^2}\Big\|_{\ell^1_k\ell^2_j}\\
&\ls\Big\|\ve_2^3\sum_{j_1,j_2,j_3\ge-1}\sum_{(k_1,k_2,k_3)\in\cY_k}
2^{-k_3(1-2\alpha/N)-2\alpha(1-7/N)\max\{j,j_1,j_2,j_3\}}
\Big\|_{\ell^1_k\ell^2_j}\ls\ve_2^3.
\end{split}
\end{equation}

\vskip 0.2 true cm
\noindent\textbf{Case 2. In the set $\cI_2:=\{j\le\log_2s+20,\max\{j_1,j_2,j_3\}\le\log_2s\}$}
\vskip 0.1 true cm

It is convenient to assume $\max\{j_1,j_2,j_3\}=j_1$.
Similarly to \eqref{Znorm3}, applying \eqref{trilinear}, \eqref{BA:d=2'} with $n=0$ and \eqref{local:Stric} with $\beta_1=\alpha+\frac{\alpha}{5N}<\beta_2=\alpha+\frac{\alpha}{4N}<9\alpha/8<1$,
$n=0,n=23/2$, respectively, we can achieve that for $N\ge12$,
\begin{equation*}
\begin{split}
&\quad\;2^{j\alpha}\|\id_{\cI_2}(s)I_{kk_1k_2k_3}^{jj_1j_2j_3}(s)\|_{L^1([1,t])L_x^2}\\
&\ls2^{2(k_1+k_2)+5k_3+j\alpha-\beta_1\max\{j,j_1\}}
\|Q_{j_1}P_{k_1}V_\mu\|_{L_t^\infty L_x^2}\\
&\qquad\times\|s^{\beta_1/2}e^{is\sigma\Lambda}P_{[[k_2]]}Q_{j_2}P_{k_2}V_\sigma\|_{L^2([1,t])L_x^\infty}
\|s^{\beta_1/2}e^{is\iota\Lambda}Q_{j_3}P_{k_3}V_\iota\|_{L^2([1,t])L_x^\infty}\\
&\ls2^{k_3(2\beta_2-1/2)-\frac{\alpha}{5N}\max\{j,j_1\}-j_1\alpha+\beta_2(j_2+j_3)
-j_2\alpha-j_3\alpha(1-\frac{23}{2N})}\ve_2^3\\
&\ls2^{-k_3(1/2-9\alpha/4)-\frac{\alpha}{5N}\max\{j,j_1\}-j_1\alpha+12j_1\alpha/N}\ve_2^3
\ls2^{-k_3(1/2-9\alpha/4)-\frac{\alpha}{5N}\max\{j,j_1,j_2,j_3\}}\ve_2^3.
\end{split}
\end{equation*}
This, together with $\alpha\in(0,1/5)$, yields
\begin{equation}\label{Znorm5}
\Big\|2^{j\alpha}\sum_{j_1,j_2,j_3\ge-1}\sum_{(k_1,k_2,k_3)\in\cY_k}
\|\id_{\cI_2}(s)I_{kk_1k_2k_3}^{jj_1j_2j_3}(s)\|_{L_x^2}\Big\|_{L^1([1,t])\ell^1_k\ell^2_j}\\
\ls\ve_2^3.
\end{equation}

\vskip 0.2 true cm
\noindent\textbf{Case 3. In the set $\cI_3:=\{j\le\log_2s+20,\max\{j_1,j_2,j_3\}\ge\log_2s\}$}
\vskip 0.1 true cm

We can also assume $\max\{j_1,j_2,j_3\}=j_1$.
As in Case 2, choosing $\beta_1=\alpha(1-\frac{1}{4N})<\beta_2=\alpha(1-\frac{1}{5N})<1$ instead, we then have
\begin{equation*}
\begin{split}
2^{j\alpha}\|\id_{\cI_2}(s)I_{kk_1k_2k_3}^{jj_1j_2j_3}(s)\|_{L^1([1,t])L_x^2}
&\ls2^{k_3(2\beta_2-1/2)+j\alpha-\beta_1j-j_1\alpha+\beta_2(j_2+j_3)-j_2\alpha
-j_3\alpha(1-\frac{23}{2N})}\ve_2^3\\
&\ls2^{-k_3(1/2-9\alpha/4)-j_1\alpha(1-\frac{47}{4N})}\ve_2^3.
\end{split}
\end{equation*}
This implies
\begin{equation}\label{Znorm6}
\Big\|2^{j\alpha}\sum_{j_1,j_2,j_3\ge-1}\sum_{(k_1,k_2,k_3)\in\cY_k}
\|\id_{\cI_3}(s)I_{kk_1k_2k_3}^{jj_1j_2j_3}(s)\|_{L_x^2}\Big\|_{L^1([1,t])\ell^1_k\ell^2_j}\\
\ls\ve_2^3.
\end{equation}
Collecting \eqref{Znorm1}--\eqref{Znorm6} shows
\begin{equation*}
\big\|2^{j\alpha}\|Q_jP_ke^{-is\Lambda}\sT_b(U_\mu,U_\sigma,U_\iota)(s)\|_{L^2(\R^2)}
\big\|_{L^1([1,t])\ell^1_k\ell^2_j}\ls\ve_2^3.
\end{equation*}
This, together with \eqref{Znorm:local} and the Minkowski inequality, ensures
\begin{equation*}
\begin{split}
&\quad\;\Big\|\w{x}^\alpha\int_1^te^{-is\Lambda}\sT_b(U_\mu,U_\sigma,U_\iota)(s)ds\Big\|_{L^2(\R^2)}\\
&\ls\|\w{x}^\alpha e^{-is\Lambda}\sT_b(U_\mu,U_\sigma,U_\iota)(s)\|_{L^1([1,t])L^2(\R^2)}\\
&\ls\big\|2^{j\alpha}\|Q_jP_ke^{-is\Lambda}\sT_b(U_\mu,U_\sigma,U_\iota)(s)\|_{L^2(\R^2)}
\big\|_{L^1([1,t])\ell^1_k\ell^2_j}\\
&\ls\ve_2^3.
\end{split}
\end{equation*}
Therefore, the proof of \eqref{Znorm} is completed.
\end{proof}

\subsection{Proof of Theorem \ref{thm2}}

\begin{proof}[Proof of Theorem \ref{thm2}]
It is concluded from \eqref{initial:d=2}, Proposition \ref{prop:high:energy}, Lemmas \ref{lem:goodunknown} and \ref{lem:low:energy}, \eqref{BA:d=2} and the localized Strichartz estimate \eqref{local:Stric} with $\beta_1=0,\beta_2=\alpha/2N\le1/8,n=11$ that
\begin{equation*}
\begin{split}
\|U(t)\|_{H^{N}(\R^2)}&\ls\ve+\ve_2^2+\ve_2\Big(\sum_{k,j\ge-1}2^{k(9+1/8)}
\|e^{is\Lambda}P_{[k-1,k+1]}Q_jP_kP_kV(s)\|_{L^2([1,t])L^\infty(\R^2)}\Big)^2\\
&\ls\ve+\ve_2^2+\ve_2^3\sum_{k,j\ge-1}
2^{-3k/4-j\alpha(1-\frac{23}{2N})}\ls\ve+\ve_2^2+\ve_2^3.
\end{split}
\end{equation*}
This, together with \eqref{initial:d=2}, \eqref{Znorm:bdry}, \eqref{profile:eqn3} and \eqref{Znorm}, shows
that there is a constant $C_2\ge1$ such that
\begin{equation}\label{thm2:pf}
\|U(t)\|_{H^{N}(\R^2)}+\|\w{x}^\alpha V(t)\|_{L^2(\R^2)}
\le C_2(\ve+\ve_2^2+\ve_2^3).
\end{equation}
Let $\ve_0=\frac{1}{16C_2^2}$ and $\ve_2=4C_2\ve$. Then for any $t\in[1,\infty)$, \eqref{thm2:pf} is improved to
\begin{equation*}
\|U(t)\|_{H^{N}(\R^2)}+\|\w{x}^\alpha V(t)\|_{L^2(\R^2)}\le3\ve_2/4.
\end{equation*}
This, together with the local existence of classical solution to \eqref{KG}, yields that \eqref{KG}
admits a unique global solution $u\in C([0,\infty),H^{N+1}(\R^2))\cap C^1([0,\infty),H^{N}(\R^2))$.

Next, we derive the scattering of the solution \eqref{scatter}.
Denote
\begin{equation*}
\begin{split}
V^\infty(x)&:=V(1,x)+i\sum_{\mu,\nu=\pm}e^{-i\Lambda}
B_{\Phi^{-1}_{\mu\nu}a_{\mu\nu}}(U_\mu,U_\nu)(1,x)\\
&\qquad+\int_1^\infty e^{-is\Lambda}\{\cC_3+i\sum_{\mu,\sigma,\iota=\pm}
\sT_b(U_\mu,U_\sigma,U_\iota)\}ds
\end{split}
\end{equation*}
and $U^\infty(t)=e^{it\Lambda}V^\infty$.
By \eqref{Znorm:bdry} and \eqref{Znorm}, $V^\infty(x)$ is in $L^2(\R^2)$.
According to \eqref{profile:eqn3}, we obtain
\begin{equation}\label{scatter:pf1}
\begin{split}
&\|U(t)-U^\infty(t)\|_{L^2}=\|e^{it\Lambda}(V(t)-V^\infty)\|_{L^2(\R^2)}\\
&\ls\sum_{\mu,\nu=\pm}\|e^{-it\Lambda}B_{\Phi^{-1}_{\mu\nu}a_{\mu\nu}}
(U_\mu,U_\nu)(t,x)\|_{L^2(\R^2)}\\
&\quad+\int_t^\infty\{\|e^{-is\Lambda}\cC_3\|_{L^2(\R^2)}+\sum_{\mu,\sigma,\iota=\pm}
\|e^{-is\Lambda}\sT_b(U_\mu,U_\sigma,U_\iota)\|_{L^2(\R^2)}\}ds.
\end{split}
\end{equation}
From \eqref{Znorm}, one knows that the last line of \eqref{scatter:pf1} tends to zero as $t\rightarrow+\infty$.
Note that \eqref{local:disp} with $n_1=0$ and $n_2=1$ leads to
\begin{equation}\label{scatter:pf2}
\begin{split}
&\quad\;\|e^{-it\Lambda}B_{\Phi^{-1}_{\mu\nu}a_{\mu\nu}}(U_\mu,U_\nu)(t,x)\|_{L^2(\R^2)}\\
&\ls\Big\|\sum_{\substack{(k_1,k_2)\in\cX_k,\\k_1\le k_2}}
\sum_{j\ge-1}2^{7k_1+k_2}\|P_{k_2}U(t)\|_{L^2(\R^2)}
\|e^{it\Lambda}P_{[k_1-1,k_1+1]}Q_jP_{k_1}V(t)\|_{L^\infty(\R^2)}\Big\|_{\ell^2_k}\\
&\ls\ve_2^2t^{-\alpha(1-1/N)}.
\end{split}
\end{equation}
Define $u_0^\infty=\Im(\Lambda^{-1}e^{i\Lambda}V^\infty)$, $u_1^\infty=\Re(e^{i\Lambda}V^\infty)$ and $u^\infty$ is the solution to the linear Klein-Gordon equation with initial data $(u_0^\infty,u_1^\infty)$ at time $t=1$, then \eqref{scatter:pf1} and \eqref{scatter:pf2} imply \eqref{scatter}.
\end{proof}

\appendix
\setcounter{equation}{1}
\section{Estimates of multilinear Fourier multipliers}

\begin{lemma}\label{lem:bilinear}
Suppose that $\Phi_{\mu\nu}$ is defined by \eqref{phase:def} and $B_m(f,g)$ is defined by \eqref{bilinear:def} with two functions $f,g$ on $\R^d$.
For any $k_1,k_2\ge-1$ and $p,q,r\in[1,\infty]$ satisfying $1/p=1/q+1/r$, it holds that
\addtocounter{equation}{1}
\begin{align}
\|B_{m_\cS({\rm or}~m_{\cS_1})}(P_{k_1}f,P_{k_2}g)\|_{L^p(\R^d)}
&\ls2^{(2d+3)\min\{k_1,k_2\}}\|P_{k_1}f\|_{L^q(\R^d)}\|P_{k_2}g\|_{L^r(\R^d)},
\tag{\theequation a}\label{bilinear:a}\\
\|B_{a_{\mu\nu}}(P_{k_1}f,P_{k_2}g)\|_{L^p(\R^d)}
&\ls2^{k_2}\|P_{k_1}f\|_{L^q(\R^d)}\|P_{k_2}g\|_{L^r(\R^d)},
\tag{\theequation b}\label{bilinear:b}\\
\|B_{\Phi^{-1}_{\mu\nu}m_\cQ}(P_{k_1}f,P_{k_2}g)\|_{L^p(\R^d)}
&\ls2^{(2d+4)k_1+2Nk_2}\|P_{k_1}f\|_{L^q(\R^d)}\|P_{k_2}g\|_{L^r(\R^d)},
\tag{\theequation c}\label{bilinear:c}\\
\|B_{\Phi^{-1}_{\mu-}m_\cQ}(P_{k_1}f,P_{k_2}g)\|_{L^p(\R^d)}
&\ls2^{k_1+(2N-1)k_2}\|P_{k_1}f\|_{L^q(\R^d)}\|P_{k_2}g\|_{L^r(\R^d)},k_1\le k_2-6,
\tag{\theequation d}\label{bilinear:d}\\
\|B_{\Phi^{-1}_{\mu\nu}a_{\mu\nu}}(P_{k_1}f,P_{k_2}g)\|_{L^p(\R^d)}
&\ls2^{(2d+3)\min\{k_1,k_2\}+k_2}\|P_{k_1}f\|_{L^q(\R^d)}\|P_{k_2}g\|_{L^r(\R^d)},
\tag{\theequation e}\label{bilinear:e}
\end{align}
where one can see \eqref{sem:energy1} for $m_\cS$, \eqref{sem:energy8} for $m_{\cS_1}$, \eqref{symbol:a} for $a_{\mu\nu}$
and \eqref{qua:energy1} for $m_\cQ$.
\end{lemma}
\begin{proof}
For convenience, we only deal with the case of $k_1\le k_2$ since the case of $k_1\ge k_2$ can be treated analogously.
According to the definition of the bilinear pseudoproduct operator \eqref{bilinear:def}, we have
\begin{equation*}
\begin{split}
&B_m(P_{k_1}f,P_{k_2}g)(x)=(2\pi)^{-2d}\iint_{(\R^d)^2}\cK(x-y,x-z)P_{k_1}f(y)P_{k_2}g(z)dydz,\\
&\cK(y,z)=\iint_{(\R^d)^2}e^{i(y\cdot\xi+z\cdot\eta)}m(\xi,\eta)
\psi_{[[k_1]]}(\xi)\psi_{[[k_2]]}(\eta)d\xi d\eta.
\end{split}
\end{equation*}
As in Lemma 3.3 of \cite{DIP17}, the $L^1$ norm of the Schwartz kernel $\cK(y,z)$ can be bounded by
\begin{equation}\label{bilinear1}
\begin{split}
\|\cK(y,z)\|_{L^1((\R^d)^2)}&\ls\|(1+|y|+|z|)^{d+1}\cK(y,z)\|_{L^2((\R^d)^2)}\\
&\ls\sum_{l=0}^{d+1}(2^{lk_1}\|\psi_{[[k_1]]}(\xi)\p_\xi^lm(\xi,\eta)\|_{L^\infty}
+2^{lk_2}\|\psi_{[[k_2]]}(\eta)\p_\eta^lm(\xi,\eta)\|_{L^\infty}).
\end{split}
\end{equation}
Inspired by Lemma 4.5 in \cite{Zheng19}, we will show that
\begin{equation}\label{bilinear2}
\sum_{l=0}^{d+1}(2^{lk_1}|\psi_{[[k_1]]}(\xi)\p^l_\xi\Phi^{-1}_{\mu\nu}(\xi,\eta)|
+2^{lk_2}|\psi_{[[k_2]]}(\eta)\p^l_\eta\Phi^{-1}_{\mu\nu}(\xi,\eta)|)\ls2^{(2d+3)k_1}
\end{equation}
and
\begin{equation}\label{bilinear3}
\sum_{l=0}^{d+1}(2^{lk_1}|\psi_{[[k_1]]}(\xi)\p^l_\xi m_\cQ(\xi,\eta)|
+2^{lk_2}|\psi_{[[k_2]]}(\eta)\p^l_\eta m_\cQ(\xi,\eta)|)\ls2^{2Nk_2+k_1}.
\end{equation}
Furthermore, if $k_1\le k_2-6$, one has
\begin{equation}\label{bilinear4}
\sum_{l=0}^{d+1}(2^{lk_1}|\psi_{[[k_1]]}(\xi)\p^l_\xi\Phi_{\mu-}^{-1}(\xi,\eta)|
+2^{lk_2}|\psi_{[[k_2]]}(\eta)\p^l_\eta\Phi_{\mu-}^{-1}(\xi,\eta)|)\ls2^{-k_2}.
\end{equation}
Meanwhile, \eqref{symbol:a} implies
\begin{equation}\label{bilinear5}
\sum_{l=0}^{d+1}(2^{lk_1}|\psi_{[[k_1]]}(\xi)\p_\xi^la_{\mu\nu}(\xi,\eta)|
+2^{lk_2}|\psi_{[[k_2]]}(\eta)\p_\eta^la_{\mu\nu}(\xi,\eta)|)\ls2^{k_2}.
\end{equation}
In fact, if \eqref{bilinear2}--\eqref{bilinear5} have been proved, then these together
with \eqref{bilinear1} and the H\"{o}lder inequality yield \eqref{bilinear:a}--\eqref{bilinear:e}.

The estimate on the first term of \eqref{bilinear2} follows from $|\p_\xi^l\Phi^{-1}_{\mu\nu}(\xi+\eta,\eta)|\ls|\Phi^{-1}_{\mu\nu}(\xi+\eta,\eta)|\ls2^{k_1}$
and direct computation. In addition,
the second term in \eqref{bilinear2} can be easily treated for the case of $k_1\ge k_2-5$.

We next treat the second term in \eqref{bilinear2} for $k_1\le k_2-6$ and $k_2\ge0$.

For $\p_\eta^l\Phi_{\mu+}$, direct computation yields
\begin{equation}\label{bilinear6}
|\p_\eta^l\Phi_{\mu+}(\xi,\eta)|=|\p_\eta^l\Lambda(\xi+\eta)-\p_\eta^l\Lambda(\eta)|
\le\int_0^1|\xi\p^{1+l}\Lambda(s\xi+\eta)|ds\ls|\xi|(1+|\eta|)^{-l},
\end{equation}
which derives $|\eta|^l|\p_\eta^l\Phi_{\mu+}(\xi,\eta)|\ls|\xi|$.
By \eqref{phi:lowerbound}, \eqref{phi:deri:bound} and Leibniz's rules, we have
\begin{equation*}
|\eta|^l|\p_\eta^l\Phi_{\mu+}^{-1}(\xi,\eta)|\ls|\xi|^{2l+1},\quad l=0,1,\cdots,d+1.
\end{equation*}
This leads to \eqref{bilinear2} for $\nu=+$.

For $\p_\eta^l\Phi_{\mu-}$, according to the definition \eqref{phase:def}, it is known that
there is a positive constant $C>0$ such that
\begin{equation*}
-\Phi_{\mu-}(\xi,\eta)=\Lambda(\xi+\eta)-\mu\Lambda(\xi)+\Lambda(\eta)
\ge\Lambda(\xi+\eta)\ge C2^{k_2}.
\end{equation*}
When $l\ge1$, we obtain
$|\p_\eta^l\Phi_{\mu-}(\xi,\eta)|=|\p_\eta^l(\Lambda(\xi+\eta)+\Lambda(\eta))|
\le|\eta|^{1-l}$.
Analogously, for $l=0,1,\cdots,d+1$, one has
$|\eta|^l|\p_\eta^l\Phi_{\mu-}^{-1}(\xi,\eta)|\ls2^{-k_2}$,
which implies \eqref{bilinear2} for $\nu=-$ and \eqref{bilinear4}.

At last, similarly to \eqref{bilinear6}, we can achieve
\begin{equation*}
\begin{split}
&n_4(\xi_1+\xi_2)n_5(\xi_2)-n_4(\frac{\xi_1+2\xi_2}{2})n_5(\frac{\xi_1+2\xi_2}{2})\\
&=\int_0^1\frac{d}{d\theta}\Big[n_4(\xi_1+\xi_2-\theta\frac{\xi_1}{2})
n_5(\xi_2+\theta\frac{\xi_1}{2})\Big]d\theta\\
&=\frac{\xi_1}{2}\int_0^1(-(\nabla n_4)n_5+n_4(\nabla n_5))d\theta,
\end{split}
\end{equation*}
which yields \eqref{bilinear3}.
\end{proof}

\begin{lemma}\label{lem:trilinear}
Suppose that $\sT_b(f,g,h)$ is defined by \eqref{trilinear:def} with three functions $f,g,h$ on $\R^2$. For any $k_1,k_2,k_3\ge-1$ and $p,q_1,q_2,q_3\in[1,\infty]$ satisfying $1/p=1/q_1+1/q_2+1/q_3$, it holds that
\begin{equation}\label{trilinear}
\|\sT_b(P_{k_1}f,P_{k_2}g,P_{k_3}h)\|_{L^p(\R^2)}\ls2^{3\max\{k_1,k_2,k_3\}+2(k_1+k_2+k_3)}
\|P_{k_1}f\|_{L^{q_1}}\|P_{k_2}g\|_{L^{q_2}}\|P_{k_3}h\|_{L^{q_3}},
\end{equation}
where $b$ is defined by \eqref{symbol:b}.
\end{lemma}
\begin{proof}
According to the definition of the trilinear pseudoproduct operator \eqref{trilinear:def}, one has
\begin{equation*}
\begin{split}
&\sT_b(P_{k_1}f,P_{k_2}g,P_{k_3}h)(x)=(2\pi)^{-6}\iiint_{(\R^2)^3}\cK(x-x_1,x-x_2,x-x_3)\\
&\hspace{5cm}\times P_{k_1}f(x_1)P_{k_2}g(x_2)P_{k_3}h(x_3)dx_1dx_2dx_3,\\
&\cK(x_1,x_2,x_3)=\iiint_{(\R^2)^3}e^{i(x_1\cdot\xi+x_2\cdot\eta+x_3\cdot\zeta)}
b(\xi,\eta,\zeta)\psi_{[[k_1]]}(\xi)\psi_{[[k_2]]}(\eta)\psi_{[[k_3]]}(\zeta)
d\xi d\eta d\zeta.
\end{split}
\end{equation*}
It follows from the method of stationary phase and \eqref{phi:deri:bound}, \eqref{symbol:b} that
\begin{equation*}
\begin{split}
(1+|x_1|+|x_2|+|x_3|)^{7}|\cK(x_1,x_2,x_3)|
&\ls2^{3\max\{k_1,k_2,k_3\}}\prod_{n=1}^3\sum_{l=0}^{7}\|\p^l\psi_{[[k_n]]}\|_{L^1(\R^2)}\\
&\ls2^{3\max\{k_1,k_2,k_3\}+2(k_1+k_2+k_3)},
\end{split}
\end{equation*}
which implies
\begin{equation*}
\|\cK(x_1,x_2,x_3)\|_{L^1((\R^2)^3)}\ls2^{3\max\{k_1,k_2,k_3\}+2(k_1+k_2+k_3)}.
\end{equation*}
This, together with the H\"{o}lder inequality, leads to \eqref{trilinear}.
\end{proof}

\section{Reformulation of the good unknown}\label{section:B}

\begin{proof}[Proof of \eqref{cU:eqn}]
At first, direct computation yields
\begin{equation}\label{B1}
\begin{split}
(\p_t-iT_{Q^{0j}\zeta_j})^2u&=(\p_t-iT_{Q^{0j}\zeta_j})(\p_t-iT_{Q^{0l}\zeta_l})u\\
&=\p_t^2u-iT_{\p_tQ^{0j}\zeta_j}u-2iT_{Q^{0j}\zeta_j}\p_tu
-T_{Q^{0j}\zeta_j}T_{Q^{0l}\zeta_l}u.
\end{split}
\end{equation}
By the definitions \eqref{error:def} and \eqref{good:unknown}, we have that
\begin{equation*}
\begin{split}
&(\p_t-iT_{Q^{0j}\zeta_j+\sqrt{1+q}\Lambda(\zeta)})\cU
=(\p_t-iT_{Q^{0j}\zeta_j}-iT_{\sqrt{1+q}\Lambda(\zeta)})
(\p_tu-iT_{Q^{0l}\zeta_l}u+iT_{\sqrt{1+q}}\Lambda u)\\
&=(\p_t-iT_{Q^{0j}\zeta_j})^2u+i(\p_t-iT_{Q^{0j}\zeta_j})T_{\sqrt{1+q}}\Lambda u
-iT_{\sqrt{1+q}\Lambda(\zeta)}(\p_tu-iT_{Q^{0j}\zeta_j}u)
+T_{\sqrt{1+q}\Lambda(\zeta)}T_{\sqrt{1+q}}\Lambda u\\
&=(\p_t-iT_{Q^{0j}\zeta_j})^2u+iT_{\p_t\sqrt{1+q}}\Lambda u
+iT_{\sqrt{1+q}}\p_t\Lambda u+T_{Q^{0j}\zeta_j}T_{\sqrt{1+q}}\Lambda u
-iT_{\sqrt{1+q}\Lambda(\zeta)}\p_tu\\
&\quad-T_{\sqrt{1+q}\Lambda(\zeta)}T_{Q^{0j}\zeta_j}u
+E(\sqrt{1+q}\Lambda(\zeta),\sqrt{1+q})\Lambda u+T_{(1+q)\Lambda(\zeta)}\Lambda u.
\end{split}
\end{equation*}
Note that Lemma \ref{lem:para} (ii) leads to
\begin{equation*}
\begin{split}
T_{\sqrt{1+q}}\p_t\Lambda u-T_{\sqrt{1+q}\Lambda(\zeta)}\p_tu
&=E(\sqrt{1+q},\Lambda(\zeta))\p_tu,\\
T_{Q^{0j}\zeta_j}T_{\sqrt{1+q}}\Lambda u
-T_{\sqrt{1+q}\Lambda(\zeta)}T_{Q^{0j}\zeta_j}u
&=E(Q^{0j}\zeta_j,\sqrt{1+q})\Lambda u
-E(\sqrt{1+q}\Lambda(\zeta),Q^{0j}\zeta_j,\Lambda^{-1}(\zeta))\Lambda u.
\end{split}
\end{equation*}
This, together with \eqref{B1} implies
\begin{equation*}
\begin{split}
&\;\quad(\p_t-iT_{Q^{0j}\zeta_j+\sqrt{1+q}\Lambda(\zeta)})\cU\\
&=\p_t^2u-iT_{\p_tQ^{0j}\zeta_j}u-2iT_{Q^{0j}\zeta_j}\p_tu
-T_{Q^{0j}\zeta_j}T_{Q^{0l}\zeta_l}u+iT_{\p_t\sqrt{1+q}}\Lambda u\\
&\quad+iE(\sqrt{1+q},\Lambda(\zeta))\p_tu+E(Q^{0j}\zeta_j,\sqrt{1+q})\Lambda u
-E(\sqrt{1+q}\Lambda(\zeta),Q^{0j}\zeta_j,\Lambda^{-1}(\zeta))\Lambda u\\
&\quad+E(\sqrt{1+q}\Lambda(\zeta),\sqrt{1+q})\Lambda u
+\Lambda^2u+T_{(Q^{jl}+Q^{0j}Q^{0l})\zeta_j\zeta_l\Lambda^{-1}(\zeta)}\Lambda u.
\end{split}
\end{equation*}
In addition, according to \eqref{nonlinear} and \eqref{remainder:def}, we have
\begin{equation}\label{B2}
\begin{split}
\p_t^2u+\Lambda^2u&=S(u,\p u)+2H(Q^{0j},\p^2_{tj}u)+2T_{\p^2_{tj}u}Q^{0j}
+2iT_{Q^{0j}}T_{\zeta_j}\p_tu\\
&\quad+H(Q^{jl},\p^2_{jl}u)+T_{\p^2_{jl}u}Q^{jl}
-T_{Q^{jl}}T_{\zeta_j\zeta_l\Lambda^{-1}(\zeta)}\Lambda u.
\end{split}
\end{equation}
Then
\begin{equation}\label{B3}
\begin{split}
&(\p_t-iT_{Q^{0j}\zeta_j+\sqrt{1+q}\Lambda(\zeta)})\cU
=S(u,\p u)+2H(Q^{0j},\p^2_{tj}u)+H(Q^{jl},\p^2_{jl}u)+2T_{\p^2_{tj}u}Q^{0j}\\
&\quad+T_{\p^2_{jl}u}Q^{jl}-iT_{\p_tQ^{0j}\zeta_j}u+iT_{\p_t\sqrt{1+q}}\Lambda u
+2iE(Q^{0j},\zeta_j)\p_tu-E(Q^{jl},\zeta_j\zeta_l\Lambda^{-1}(\zeta))\Lambda u\\
&\quad-E(Q^{0j}\zeta_j,Q^{0l}\zeta_l,\Lambda^{-1}(\zeta))\Lambda u
+iE(\sqrt{1+q},\Lambda(\zeta))\p_tu+E(Q^{0j}\zeta_j,\sqrt{1+q})\Lambda u\\
&\quad-E(\sqrt{1+q}\Lambda(\zeta),Q^{0j}\zeta_j,\Lambda^{-1}(\zeta))\Lambda u
+E(\sqrt{1+q}\Lambda(\zeta),\sqrt{1+q})\Lambda u.
\end{split}
\end{equation}
For the terms $T_{\p_tQ^{0j}\zeta_j}u$ and $T_{\p_t\sqrt{1+q}}\Lambda u$
in the second line of \eqref{B3},
due to $\p_t^2u=\Delta u-u+F(u,\p u,\p\p_xu)$, one can find that
\begin{equation*}
\begin{split}
\p_tQ^{0j}&=\cF_0(u,\p u)\p_t^2u+\cF_1(u,\p u,\p\p_xu)+\cF_2(u,\p u,\p\p_xu)\\
&=\cF_1(u,\p u,\p\p_xu)+\cF_2(u,\p u,\p\p_xu),\\
\p_t\sqrt{1+q}&=\frac12(1+q)^{-1/2}(\p_tQ^{jl}+2\p_tQ^{0j}Q^{0l})
\zeta_j\zeta_l\Lambda^{-2}(\zeta)\\
&=(1+q)^{-1/2}(\cF_1(u,\p u,\p\p_xu)+\cF_2(u,\p u,\p\p_xu)),
\end{split}
\end{equation*}
where $\cF_1(0,0,0)=0$, $\cF_1(u,\p u,\p\p_xu)$ is linear in $(u,\p u,\p\p_xu)$ and $\cF_2(u,\p u,\p\p_xu)$ is at least second order of $(u,\p u,\p\p_xu)$.
Therefore,
\begin{equation}\label{B4}
\begin{split}
T_{\p_tQ^{0j}\zeta_j}u&=T_{(\cF_1+\cF_2)\zeta_j\Lambda^{-1}(\zeta)}\Lambda u
+E((\cF_1+\cF_2)\zeta_j,\Lambda^{-1}(\zeta))\Lambda u,\\
T_{\p_t\sqrt{1+q}}\Lambda u&=T_{\cF_1}\Lambda u
+T_{((1+q)^{-1/2}-1)\cF_1+(1+q)^{-1/2}\cF_2}\Lambda u.
\end{split}
\end{equation}
Inserting \eqref{B4} into \eqref{B3} with the fact $E(1,a)f=E(a,1)f=0$ yields \eqref{cU:eqn}.
\end{proof}

\begin{remark}\label{fully:nonl}
Suppose that the nonlinearity \eqref{nonlinear} has the form
\begin{equation*}
F(u,\p u,\p\p_xu)=Q(u,\p u,\p\p_xu)+2\sum_{j,k=1}^dF^{0j0k}\p_{0j}^2u\p_{0k}^2u
+\sum_{j,k,l,m=1}^dF^{jklm}\p_{jk}^2u\p_{lm}^2u+R(u,\p u,\p\p_xu),
\end{equation*}
where $Q(u,\p u,\p\p_xu)$ is quadratic and linear in $\p\p_xu$, $R(u,\p u,\p\p_xu)$ is cubic, $F^{0j0k}=F^{0k0j}$ and $F^{jklm}=F^{lmjk}$.
It only suffices to deal with the quadratic term of $\p\p_xu$ since the higher order term $R(u,\p u,\p\p_xu)$ can be treated similarly.
As in \eqref{B2}, we have
\begin{equation*}
\begin{split}
F(u,\p u,\p\p_xu)&=Q(u,\p u,\p\p_xu)+\sum_{j,k=1}^d
\Big(4F^{0j0k}T_{\p_{0j}^2u}\p_{0k}^2u+2H(\p_{0j}^2u,\p_{0k}^2u)\Big)\\
&\quad+\sum_{j,k,l,m=1}^d\Big(2F^{jklm}T_{\p_{jk}^2u}\p_{lm}^2u
+H(\p_{jk}^2u,\p_{lm}^2u)\Big)+R(u,\p u,\p\p_xu).
\end{split}
\end{equation*}
Therefore, for the fully nonlinear quadratic $F(u,\p u,\p^2u)$ in \eqref{KG},
Theorems \ref{thm1}-\ref{thm2} can be established analogously.
\end{remark}

\vskip 0.5 true cm
{\bf \color{blue}{Conflict of interest}}
\vskip 0.2 true cm

{\bf On behalf of all authors, the corresponding author states that there is no conflict of interest.}

\vskip 0.3 true cm

{\bf \color{blue}{Data availability}}

\vskip 0.2 true cm

{\bf Data sharing is not applicable to this article as no new data were created.}


\begin{thebibliography}{99}

\bibitem{Bambusi03} D. Bambusi, {\it Birkhoff normal form for some nonlinear PDEs,} Comm. Math. Phys. \textbf{234} (2003), no. 2, 253--285.

\bibitem{BG06} D. Bambusi, B. Gr\'{e}bert, {\it Birkhoff normal form for partial differential equations with tame modulus,} Duke Math. J. \textbf{135} (2006), no. 3, 507--567.

\bibitem{Bourgain96} J. Bourgain, {\it Construction of approximative and almost periodic solutions of perturbed linear Schr\"{o}dinger and wave equations,} Geom. Funct. Anal. \textbf{6} (1996), no. 2, 201--230.



\bibitem{Delort97} J.M. Delort, {\it Sur le temps d'existence pour l'\'{e}quation de Klein-Gordon semi-lin\'{e}aire en dimension 1. (French) [Existence time for the one-dimensional semilinear Klein-Gordon equation]} Bull. Soc. Math. France \textbf{125} (1997), no. 2, 269--311.

\bibitem{Delort98} J.M. Delort, {\it Temps d'existence pour l'\'{e}quation de Klein-Gordon semi-lin\'{e}aire \`{a} donn\'{e}es petites p\'{e}riodiques. (French) [Time of existence for the semilinear Klein-Gordon equation with periodic small data]} Amer. J. Math. \textbf{120} (1998), no. 3, 663--689.

\bibitem{Delort01} J.M. Delort, {\it Existence globale et comportement asymptotique pour l'\'{e}quation de Klein-Gordon quasi lin\'{e}aire \`{a} donn\'{e}es petites en dimension 1. (French) [Global existence and asymptotic behavior for the quasilinear Klein-Gordon equation with small data in dimension 1]} Ann. Sci. \'{E}cole Norm. Sup. (4) \textbf{34} (2001), no. 1, 1--61.

\bibitem{Delort09} J.M. Delort, {\it On long time existence for small solutions of semi-linear Klein-Gordon equations on the torus,} J. Anal. Math. \textbf{107} (2009), 161--194.

\bibitem{DF00} J.M. Delort, Daoyuan Fang, {\it Almost global existence for solutions of semilinear Klein-Gordon equations with small weakly decaying Cauchy data,} Comm. Partial Differential Equations \textbf{25} (2000), no. 11-12, 2119--2169.

\bibitem{DS04} J.M. Delort, J. Szeftel, {\it Long-time existence for small data nonlinear Klein-Gordon equations on tori and spheres,} Int. Math. Res. Not. 2004, no. \textbf{37}, 1897--1966.

\bibitem{DIP17} Yu Deng, A. D. Ionescu, B. Pausader, {\it The Euler-Maxwell system for electrons: global solutions in 2D,} Arch. Ration. Mech. Anal. \textbf{225} (2017), no. 2, 771--871.

\bibitem{FH20} L. Forcella, L. Hari, {\it Large data scattering for NLKG on waveguide $\mathbb{R}^d\times\mathbb{T}$,} J. Hyperbolic Differ. Equ. \textbf{17} (2020), no. 2, 355--394.

\bibitem{GMP13} P. Germain, N. Masmoudi, B. Pausader, {\it Nonneutral global solutions for the electron Euler-Poisson system in three dimensions,} SIAM J. Math. Anal. \textbf{45} (2013), no. 1, 267--278.

\bibitem{Guo98} Yan Guo, {\it Smooth irrotational flows in the large to the Euler-Poisson system in $\mathbb{R}^{3+1}$}, Comm. Math. Phys. \textbf{195} (2) (1998), 249--265.

\bibitem{HV18} L. Hari, N. Visciglia, {\it Small data scattering for energy critical NLKG on product spaces $\mathbb{R}^d\times\mathcal{M}^2$,} Commun. Contemp. Math. \textbf{20} (2018), no. 2, 1750036, 11 pp.


\bibitem{IP13} A. D. Ionescu, B. Pausader, {\it The Euler-Poisson system in 2D: global stability of the constant equilibrium solution,} Int. Math. Res. Not. IMRN 2013, no. 4, 761--826.

\bibitem{IP14} A. D. Ionescu, B. Pausader, {\it Global solutions of quasilinear systems of Klein-Gordon equations in 3D,} J. Eur. Math. Soc. (JEMS) \textbf{16} (2014), no. 11, 2355--2431.

\bibitem{Klainerman85} S. Klainerman, {\it Global existence of small amplitude solutions to nonlinear Klein-Gordon equations in four space-time dimensions,} Comm. Pure Appl. Math. \textbf{38} (1985), no. 5, 631--641.

\bibitem{LW14} Dong Li, Yifei Wu, {\it The Cauchy problem for the two dimensional Euler-Poisson system,} J. Eur. Math. Soc. (JEMS) \textbf{16} (2014), no. 10, 2211--2266.

\bibitem{LTY22a} Jun Li, Fei Tao, Huicheng Yin, {\it Almost global smooth solutions of the 3D quasilinear Klein-Gordon equations on the product space $\mathbb{R}^2\times\mathbb{T}$,} arXiv:2204.08130 (2022)

\bibitem{LTY22b} Jun Li, Fei Tao, Huicheng Yin, {\it On global smooth small data solutions of 3-D quasilinear Klein-Gordon equations on $\mathbb{R}^2\times\mathbb{T}$,} Preprint (2022)

\bibitem{NS11book} K. Nakanishi, W. Schlag, {\it Invariant manifolds and dispersive Hamiltonian evolution equations,} Zurich Lectures in Advanced Mathematics. European Mathematical Society (EMS), Z\"{u}rich, vi+253 pp. (2011)

\bibitem{OTT96} T. Ozawa, K. Tsutaya, Y. Tsutsumi, {\it Global existence and asymptotic behavior of solutions for the Klein-Gordon equations with quadratic nonlinearity in two space dimensions,} Math. Z. \textbf{222} (1996), no. 3, 341--362.

\bibitem{Shatah85} J. Shatah, {\it Normal forms and quadratic nonlinear Klein-Gordon equations,} Comm. Pure Appl. Math. \textbf{38} (1985), no. 5, 685--696.

\bibitem{ST93} J.C.H. Simon, E. Taflin, {\it The Cauchy problem for nonlinear Klein-Gordon equations,} Comm. Math. Phys. \textbf{152} (1993), no. 3, 433--478.

\bibitem{Stein} E. M. Stein, {\it Harmonic Analysis: Real-Variable Methods, Orthogonality, and Oscillatory Integrals,} Princeton University Press, Princeton, 1993.

\bibitem{Stingo18} A. Stingo, {\it Global existence and asymptotics for quasi-linear one-dimensional Klein-Gordon equations with mildly decaying Cauchy data,} Bull. Soc. Math. France \textbf{146} (2018), no. 1, 155--213.

\bibitem{TY21} Fei Tao, Huicheng Yin, {\it Global smooth solutions of the 4-D quasilinear Klein-Gordon equations on the product space $\mathbb{R}^3\times\mathbb{T}$,} J. Differential Equations \textbf{352} (2023), 67--121.


\bibitem{Zheng19} Fan Zheng, {\it Long-term regularity of the periodic Euler-Poisson system for electrons in 2D,} Comm. Math. Phys. \textbf{366} (2019), no. 3, 1135--1172.

\bibitem{Zheng22} Fan Zheng, {\it Long-term regularity of 3D gravity water waves,} Comm. Pure Appl. Math. \textbf{75} (2022), no. 5, 1074--1180.
\end{thebibliography}
\end{document}